\documentclass[10pt]{article}
\usepackage{geometry, enumerate, amsmath, amssymb, amsthm, amscd, hyperref, graphicx, color, enumerate, mathrsfs}
\usepackage[all]{xy}

\title{Stacked Pseudo-Convergent Sequences and Polynomial Dedekind Domains}

\date{\today}

\author{Giulio Peruginelli\footnote{Department of Mathematics, University of Padova, Via Trieste, 63, 35121 Padova, Italy.\newline E-mail: gperugin@math.unipd.it}
}

\numberwithin{equation}{section}

\newtheorem{Prop}[equation]{Proposition}
\newtheorem{Thm}[equation]{Theorem}
\newtheorem{Lem}[equation]{Lemma}
\newtheorem{Cor}[equation]{Corollary}

\theoremstyle{definition}\newtheorem{Def}[equation]{Definition}

\newtheorem{Rem}[equation]{Remark}
\newtheorem{Rems}[equation]{Remarks}
\theoremstyle{definition}

\newcommand{\F}{\mathbb{F}}
\newcommand{\Q}{\mathbb{Q}}
\newcommand{\N}{\mathbb{N}}
\newcommand{\Z}{\mathbb{Z}}
\newcommand{\R}{\mathbb{R}}
\renewcommand{\O}{\mathbb{O}}
\newcommand{\C}{\mathbb{C}}
\newcommand{\PP}{\mathbb{P}}
\newcommand{\Int}{\textnormal{Int}}
\newcommand{\IntQ}{\Int_{\Q}}

\newcommand{\uE}{\underline{E}}
\newcommand{\oQp}{\overline{\Q_p}}
\newcommand{\Br}{\text{Br}}
\newcommand{\Gal}{\text{Gal}}

\newcommand{\K}{\widehat{K}}
\newcommand{\V}{\widehat{V}}

\newcommand{\wW}{\widetilde{W}}

\newcommand{\oZp}{\overline{\Z_p}}

\newcommand{\hZ}{\widehat{\mathbb{Z}}}
\newcommand{\ohZ}{\overline{\hZ}}
\newcommand{\Cl}{\text{Cl}}

\begin{document}

\maketitle

\begin{abstract}
Let $p\in\Z$ be a prime, $\oQp$ a fixed algebraic closure of the field of $p$-adic numbers and $\oZp$ the absolute integral closure of the ring of $p$-adic integers. Given a residually algebraic torsion extension $W$ of $\Z_{(p)}$ to $\Q(X)$, by Kaplansky's characterization of immediate extensions of valued fields, there exists a pseudo-convergent sequence of transcendental type $E=\{s_n\}_{n\in\N}\subset\oQp$  such that $W=\Z_{(p),E}=\{\phi\in\Q(X)\mid\phi(s_n)\in\oZp,\text{ for all sufficiently large }n\in\N\}$. We show here that we may assume that $E$ is stacked, in the sense that, for each $n\in\N$, the residue field (the value group, respectively) of  $\oZp\cap\Q_p(s_n)$ is contained in the residue field (the value group, respectively) of $\oZp\cap\Q_p(s_{n+1})$; this property of $E$ allows us to describe the residue field and value group of $W$. 
In particular, if $W$ is a DVR, then there exists $\alpha$ in the completion $\C_p$ of $\oQp$,  $\alpha$ transcendental  over $\Q$, such that $W=\Z_{(p),\alpha}=\{\phi\in\Q(X)\mid\phi(\alpha)\in \O_p\}$, where $\O_p$ is the unique local ring of $\C_p$; $\alpha$ belongs to $\oQp$ if and only if the residue field extension $W/M\supseteq\Z/p\Z$ is finite. As an application, we provide a full  characterization of the Dedekind domains between $\Z[X]$ and $\Q[X]$.

\noindent Keywords: pseudo-convergent sequence, residually algebraic extension, distinguished pair, minimal pair, Dedekind domain.\\

\noindent MSC Primary 12J20, 13F30, 13A18, 13F05, 13B25, 13F20.

\end{abstract}




\subsection*{Introduction}

The problem of characterizing the set of the extensions of a valuation domain $V$ with quotient field $K$ to the field of rational functions $K(X)$ has a long and rich tradition (for example, see \cite{AlePop, APZTheorem, APZMinimal,APZAll, Kap,MatOhm,PerTransc, PS1,PS2}). One of the recent direction of research is to describe these extensions by means of pseudo-monotone sequences of $K$ (\cite{PS2}), in the original spirit of Ostrowski (see \cite{Ostr}) who introduced the well-known notion of pseudo-convergent sequence, later on expanded by Kaplansky in \cite{Kap} to study immediate extensions of valued fields. 

Here, given a prime $p\in\Z$ and the DVR $\Z_{(p)}$ of $\Q$, we are interested in describing residually algebraic torsion extensions of $\Z_{(p)}$ to $\Q(X)$, that is, valuation domains $W$ of $\Q(X)$ lying above $\Z_{(p)}$ such that the residue field extension $W/M\supseteq\Z/p\Z$ is algebraic and the value group $\Gamma_w$ of the associated valuation $w$ to $W$ is contained in the divisible hull of the value group of $\Z_{(p)}$ (i.e., the rationals). These valuation domains arise naturally as overrings of rings of integer-valued polynomials and Dedekind domains between $\Z[X]$ and $\Q[X]$ (see \cite{EakHei, PerDedekind}) and also in the description of closed subfields of $\C_p$ (\cite{IovZah}), the completion of an algebraic closure $\oQp$ of the field of $p$-adic numbers $\Q_p$. In the case $W$ is a DVR and the residue field extension is finite, then by \cite[Theorem 2.5 \& Proposition 2.2]{PerTransc} there exists an element $\alpha$ in $\oQp$, transcendental over $\Q$, such that $W=\Z_{(p),\alpha}=\{\phi\in \Q(X)\mid\phi(\alpha)\in\oZp\}$, where $\oZp$ is the absolute integral closure of $\Z_p$ (i.e., the integral closure of $\Z_p$ in $\oQp$; note that $\oZp$ is the valuation domain of the unique extension of $v_p$ to $\oQp$).  
In general, given a residually algebraic torsion extension $W$ of $\Z_{(p)}$ to $\Q(X)$, there exists a pseudo-convergent sequence $E=\{s_n\}_{n\in\N}$ in $\oQp$ such that $W=\Z_{(p),E}=\{\phi\in\Q(X)\mid\phi(s_n)\in\oZp,\text{ for all sufficiently large }n\in\N\}$ (Proposition \ref{ratQ(X)}). One of the main result of this paper is to show that we may assume that $E$ is stacked (in a sense we make clear in \S \ref{stacked sect}; see Theorem \ref{Epcvtranscdistinguished}).  In particular, if $W$ is a DVR of $\Q(X)$ extending $\Z_{(p)}$ such that the extension of the residue fields is infinite algebraic, then there exists $\alpha$ in $\C_p\setminus \oQp$, such that  $W=\Z_{(p),\alpha}=\{\phi\in \Q(X)\mid\phi(\alpha)\in \O_p\}$, where $\O_p$ is the completion of $\oZp$ (equivalently, $\O_p$ is the valuation domain of the unique extension of $v_p$ to $\C_p$; see  Corollary \ref{DVRQ(X)rat}). Necessarily, the (transcendental) extension $\Q_p(\alpha)/\Q_p$ has finite ramification. 

It is worth recalling that in \cite[\S 5, 1., \& Theorem 5.1]{APZAll} a residually algebraic torsion extension $W$ of $\Z_{(p)}$ to $\Q(X)$ is realized as the limit of a sequence of residually transcendental extensions $W_n$ of $\Z_{(p)}$ to $\Q(X)$ (i.e., the residue field extension of $W_n$ over $\Z_{(p)}$ is transcendental); moreover, for each $n\in\N$, $W_n$ is defined by a minimal pair $(s_n,\delta_n)$ (as explained in \cite[p. 282]{APZAll}; for the definition of minimal pair see \S \ref{dist and min pairs}). Here, $W$ is realized as the valuation domain $\Z_{(p),E}$, where for each $n\in\N$, $(s_n,\delta_n=v_p(s_{n+1}-s_n))$ is a minimal pair, too.

The motivations behind these results are based on the paper \cite{APZclosed}, in which the authors study closed subfields of $\C_p$ and show that any transcendental element of $\C_p$ is the limit of a particular kind of Cauchy sequence in $\oQp$ called distinguished (\cite[Proposition 2.2]{APZclosed}) which allows them to associate to such an element a set of invariants (\cite[Remark 2.4]{APZclosed}). The notion of stacked sequence we introduce in this paper is a generalization of the notion of distinguished sequence and falls into the well-known class of pseudo-convergent sequences. It allows us to describe the whole class of residually algebraic torsion extensions of $\Z_{(p)}$ to $\Q(X)$, which strictly comprise the valuation domains $\Z_{(p),\alpha}$ arising from elements $\alpha\in\C_p\setminus\oQp$.

As an application of the above results, we are able to complete the classification of the family of Dedekind domains $R$ between $\Z[X]$ and $\Q[X]$ started in \cite{PerDedekind}. In that paper we described the Dedekind domains of this family whose residue fields of prime characteristic are finite fields (\cite[Theorem 2.17]{PerDedekind}); the description is obtained by means of the notion of rings of integer-valued polynomials over algebras. We also showed that, given a group $G$ which is the direct sum of a countable family of finitely generated abelian groups, there exists a Dedekind domain $R$ with finite residue fields of prime characteristic, $\Z[X]\subset R\subseteq\Q[X]$, with class group $G$ (\cite[Theorem 3.1]{PerDedekind}).  


The paper is organized as follows. In Section \ref{prelim} we recall the relevant notions we need in our paper: first, we review the definition of pseudo-convergent sequence of a valued field $K$ and the valuation domain of $K(X)$ associated to such a sequence in the spirit of Ostrowski (\cite{Ostr}), as developed recently in \cite{PS1,PS2}. Then, we recall the notion of distinguished pair introduced in \cite{PZstructure} which later on was used in \cite{APZclosed} to describe closed subfield of $\C_p$ in terms of a specific kind of pseudo-convergent Cauchy sequence called distinguished.

In Section \ref{stacked sect}, we introduce the notion of stacked sequence $E=\{s_n\}_{n\in\N}$ in $\oQp$, which turns out to be a pseudo-convergent sequence  of transcendental type such that, for each $n\in\N$, the value group (the residue field, respectively) of $\oZp\cap\Q_p(s_n)$ is contained in the value group (the residue field, respectively) of $\oZp\cap\Q_p(s_{n+1})$. By Theorem \ref{Epcvtranscdistinguished}, every residually algebraic extension $W$ of $\Z_p$ to $\Q_p(X)$ can be realized by means of a stacked sequence $E\subset\oQp$, that is, $W=\Z_{p,E}=\{\phi\in\Q_p(X)\mid\phi(s_n)\in\oZp,\text{ for all sufficiently large }n\in\N\}$. Moreover, the above specific property of stacked sequences is crucial for the description of the residue field and value group of $W$ as the union of the ascending chain of residue fields and value groups of $\oZp\cap\Q_p(s_n)$, respectively  (Proposition \ref{residuevaluerat}).  We mentioned above that the elements $\alpha\in\C_p\setminus\oQp$ such that the extension $\Q_p(\alpha)/\Q_p$ has finite ramification give raise to DVRs of $\Q(X)$; we characterize such elements as the limits of sequences contained in the maximal unramified extension of a finite extension of $\Q_p$ (Proposition \ref{elements bounded ramification}). We close this section by pointing out a wrong statement in the paper \cite{IovZah}, namely, that the completion of $\Q_p(X)$ with respect to a residually algebraic torsion extension $W$ of $\Z_p$ is a subfield of $\C_p$; this is not true in general and  it depends on whether the above sequence $E$ is Cauchy or not. In \S \ref{RatZpQX}, we use the result of \S \ref{stackedseq} about residually algebraic torsion extensions of $\Z_p$ to $\Q_p(X)$ in order to characterize the analogous extensions of $\Z_{(p)}$ to $\Q(X)$ (Proposition \ref{ratQ(X)}). In Theorem \ref{prescribed residue field and value group rat Q(X)}, we show that, for any prescribed algebraic extension $k$ of $\F_p$ and value group $\Gamma$, $\Z\subseteq\Gamma\subseteq\Q$, there exists  $\alpha\in\C_p$, transcendental over $\Q$, such that $\Z_{(p),\alpha}$ has residue field $k$ and value group $\Gamma$.

Finally, in Section \ref{PolDed} we provide the aforementioned classification of the Dedekind domains between $\Z[X]$ and $\Q[X]$ by means of the notion of ring of integer-valued polynomials over an algebra: given such a domain $R$, we show that for each prime  $p\in\Z$ there exists a finite set $E_p\subset\C_p$  of transcendental elements over $\Q$ such that $R=\{f\in\Q[X] \mid f(E_p)\subseteq \O_p,\forall p\in\PP\}$ (Theorem \ref{PolDedekind}).

\section{Preliminaries}\label{prelim}

We refer to \cite{Bourb, EngPre, RibVal, ZS2} for generalities about valuation theory. A valuation domain $W$ of the field of rational functions $K(X)$ is an extension of  a valuation domain $V$ of $K$ if $W\cap K=V$. We denote by $w$ a valuation  associated to $W$, by $\Gamma_w$ the value group of $w$ and by $k_w$ the residue field of $W$. We recall that an extension $W$ of $V$ to $K(X)$ is called residually algebraic  if the residue field extension is algebraic and it is called torsion if $\Gamma_w$ is contained in the divisible hull of the value group $\Gamma_v$ of $V$ (see \cite{APZAll}). Given a valuation domain $W$ with quotient field $F$, a subfield $K$ of $F$ and the valuation domain $V=W\cap K$, we say that $W$ is an immediate extension of $V$ (or simply immediate over $V$) if the value groups (the residue fields, respectively) of $V$ and $W$ are the same. Given a field $K$ with a valuation domain $V$, we denote by $\widehat{K}$ ($\widehat V$, respectively) the completion of $K$ ($V$, respectively) with respect to $V$-adic topology.

\subsection{Pseudo-convergent sequences}\label{pcv}

The following basic material about pseudo-convergent sequences can be found for example in \cite{Kap, PS1,PS2}.

Given a valued field $(K,v)$, a  sequence $E=\{s_n\}_{n\in\N}\subset K$ is said to be \emph{pseudo-convergent} if, for all $n<m<k$ we have
$$v(s_n-s_m)<v(s_m-s_k).$$
In particular, for all $n$ and $m>n$ we have $v(s_n-s_m)=v(s_n-s_{n+1})$. For each $n\in\N$, we set $\delta_n=v(s_n-s_{n+1})$. The strictly increasing sequence $\{\delta_n\}_{n\in\N}$ of the value group $\Gamma_v$ of $v$ is called the \emph{gauge} of $E$. The sequence $E$ is a classical Cauchy sequence in $K$ if and only if the gauge of $E$ is cofinal in $\Gamma_v$. In this case, $E$ converges to a unique limit $\alpha\in\K$. In general, if $E=\{s_n\}_{n\in\N}\subset K$ is a pseudo-convergent sequence, we say that an element $\alpha\in K$ is a \emph{pseudo-limit} of $E$ if $v(s_n-\alpha)$ is a strictly increasing sequence. Equivalently, $v(s_n-\alpha)=\delta_n$ for each $n\in\N$. The set of pseudo-limits $\mathcal{L}_E$ in $K$ of a pseudo-convergent sequence $E$ is equal to $\mathcal{L}_E=\alpha+\Br(E)$ (\cite[Lemma 3]{Kap}), where $\Br(E)=\{x\in K\mid v(x)>\delta_n,\forall n\in\N\}$ is a fractional ideal, called the \emph{breadth ideal} of $E$. Clearly, $E$ is a Cauchy sequence if and only if $\Br(E)=\{0\}$.

As in \cite[Definitions, p. 306]{Kap}, a pseudo-convergent sequence $E=\{s_n\}_{n\in\N}\subset K$ is of \emph{transcendental type} if, for all $f\in K[X]$, $v(f(s_n))$ is eventually constant. Otherwise, $E$ is said to be of \emph{algebraic type} if $v(f(s_n))$ is eventually strictly increasing for some $f\in K[X]$. The sequence $E$ is of algebraic type if and only if, for some extension $u$ of $v$ to the algebraic closure $\overline{K}$ of $K$, there exists $\alpha\in\overline{K}$ which is a pseudo-limit of $E$ with respect to $u$. If $F$ is a subfield of $K$, then we say that $E$ is of \emph{transcendental type over} $F$ if, for all $f\in F[X]$, $v(f(s_n))$ is eventually constant. Almost all the pseudo-convergent sequences considered in this paper in order to describe residually algebraic torsion extensions to the field of rational functions are of transcendental type.

Given a pseudo-convergent sequence $E=\{s_n\}_{n\in\N}\subset K$, the following is a valuation domain of $K(X)$ extending $V$ associated to $E$ (\cite[Theorem 3.8]{PS1}):
$$V_E=\{\phi\in K(X)\mid \phi(s_n)\in V,\text{ for all sufficiently large }n\in\N\}.$$
Moreover, by the same Theorem, $X$ is a pseudo-limit of $E$  with respect to the valuation $v_E$ associated to $V_E$, so, in particular, $v_E(X-s_n)=\delta_n$, for every $n\in\N$. Also, if $E$ is of transcendental type, then for all $f\in K[X]$, we have  $v_E(f)=v(f(s_n))$ for all $n$ sufficiently large (\cite[Theorem 2]{Kap} or \cite[Theorem 4.9, a)]{PS1}).

In case $E$ is a Cauchy sequence converging to $\alpha\in\K$, then 
$$V_E=V_{\alpha}=\{\phi\in K(X)\mid \phi(\alpha)\in \V\}$$
(see \cite[Remark 3.10]{PS1}).

Given two pseudo-convergent sequences $E=\{s_n\}_{n\in\N},E'=\{s_n'\}_{n\in\N}\subset K$, we say that $E,E'$ are equivalent if $\Br(E)=\Br(E')$ and for each $k\in\N$, there exist $i_0,j_0\in\N$ such that $v(s_i-s_j')>v(s_{k+1}-s_k)$ for each $i\geq i_0$ and $j\geq j_0$  (see   \cite[\S 5]{PS1}). By \cite[Proposition 5.3]{PS2}, $E,E'$ are equivalent if and only if $V_E=V_{E'}$.

\subsection{Distinguished pairs}\label{dist and min pairs}
We suppose in this section that $(K,v)$ is a complete valued field where $v$ is a rank one discrete valuation (so, in particular, $(K,v)$ is Henselian). Let $\overline{K}$ be a fixed algebraic closure and let $v$ denote the unique extension of $v$ to $\overline{K}$. Let also $\Gamma_{\overline{v}}=\Gamma_v\otimes\Q$ be the divisible hull of $\Gamma_v$. Given an element $a\in\overline{K}$, let $O_a, k_a,\Gamma_a$ be the valuation domain of the restriction of $v$ to $K(a)$, the residue field of $O_a$ and the value group of $O_a$,   respectively.

As in \cite{KhaSa}, given $a\in\overline{K}\setminus K$, we set:
\begin{align*}
\delta_K(a)=&\sup\{v(a-c)\mid c\in\overline{K},[K(c):K]<[K(a):K]\}\\
\omega_K(a)=&\sup\{v(a-a')\mid a'\not=a\text{ runs over the }K\text{-conjugates of }a\}
\end{align*}
The following is the well-known Krasner's lemma. Essentially, given a separable element $a\in\overline{K}$, if another element $b\in\overline{K}$ is closer to $a$ than to any other of its conjugates, then $K(a)$ is a subfield of $K(b)$.
\begin{Lem}[Krasner]
If $a\in\overline{K}^{\text{sep}}$ and $b\in\overline{K}$ is such that $v(a-b)>\omega_K(a)$, then $K(a)\subseteq K(b)$.
\end{Lem}
In particular, for every $a\in\overline{K}^{\text{sep}}$ we have $\delta_K(a)\leq \omega_K(a)$. Moreover, it follows also that $\delta_K(a)$ is a maximum, since $v$ is supposed to be discrete. This is known (see for example \cite[p. 105]{PZstructure}), but for the sake of the reader we give a short proof.

\begin{Lem}\label{deltaK maximum}
In the above setting, $\delta_K(a)=\max\{v(a-c)\mid c\in\overline{K},[K(c):K]<[K(a):K]\}$.
\end{Lem}
\begin{proof}
By Krasner's Lemma, for each of the relevant $c$'s we have  $v(a-c)\leq \omega_K(a)$. Note that the ramification index of $K(a,c)$ over $K$ is (strictly) bounded by $[K(a):K]^2$. In particular, since the value $v(a-c)$ belongs to $\Gamma_{a-c}$, it follows that there exists $N\in\N$, independent from each of the above $c$'s, such that $Nv(a-c)\in\Gamma_v\cong\Z$. Hence, the set $\{v(a-c)\mid c\in\overline{K},[K(c):K]<[K(a):K]\}$ (which is a subset of $\Gamma_{\overline{v}}$) is bounded from above and its elements have bounded torsion. It follows that this set has a maximum, which is equal to $\delta_K(a)$ by its very definition.
\end{proof}

Similar to Krasner's lemma, we have the following fundamental principle (see \cite[Theorem 1.1]{KhaSa}), first discovered in \cite{PZstructure}.
\begin{Thm}\label{propertydistinguished}
Suppose that $a,b\in\overline{K}$ are such that $v(a-b)>\delta_K(b)$. Then:
\begin{itemize}
\item[i)] $\Gamma_{b}\subseteq\Gamma_{a}$.
\item[ii)] $k_{b}\subseteq k_{a}$.
\item[iii)] $[K(b):K]\mid [K(a):K]$.
\end{itemize}
\end{Thm}
Next, we recall the definition of distinguished pair introduced in \cite[p. 105]{PZstructure}.
\begin{Def}
A pair of elements $(b,a)\in\overline{K}^2$ is said to be \emph{distinguished} if the following hold:
\begin{itemize}
\item[i)] $[K(b):K]<[K(a):K]$.
\item[ii)] for all $c\in\overline{K}$ such that $[K(c):K]<[K(a):K]$ then $v(a-c)\leq v(a-b)$.
\item[iii)] for all $c\in\overline{K}$ such that $[K(c):K]<[K(b):K]$ then $v(a-c)< v(a-b)$.
\end{itemize}
\end{Def}
Part of the definition of distinguished pair is related to the notion of minimal pair, which we now recall (see for example \cite{APZAll, APZTheorem, APZMinimal}). 
\begin{Def}
Let $(a,\delta)\in\overline{K}\times\Gamma_{\overline{v}}$. We say that $(a,\delta)$ is a \emph{minimal pair} if for every $c\in\overline{K}$ such that $[K(c):K]<[K(a):K]$ we have $v(a-c)<\delta$.
\end{Def}
In other words, $(a,\delta)$ is a minimal pair if for every $b\in B(a,\delta)=\{x\in\overline{K}\mid v(a-x)\geq\delta\}$, we have $[K(b):K]\geq[K(a):K]$ (i.e., $a$ is a 'center' of the ball $B(a,\delta)$ of minimal degree). By Lemma \ref{deltaK maximum}, $(a,\delta)$ is a minimal pair if and only if $\delta>\delta_K(a)$. In particular, if $\delta>\omega_K(a)$, then $(a,\delta)$ is a minimal pair.

\begin{Rems}\label{properties distinguised pair}
Let $(b,a)$ be a distinguished pair. Note that i) and ii) imply that $v(a-b)=\delta_K(a)$. In fact, by i) and ii), it immediately follows that the inequality $\leq$ holds. Conversely, by ii) we also have that $v(a-b)\geq v(a-c)$ for all $c$ such that $[K(c):K]<[K(a):K]$, that is, $v(a-b)\geq \delta_K(a)$.

Note that iii) is equivalent to the following:
\begin{itemize}
\item[iii')] for all $c\in\overline{K}$ such that $[K(c):K]<[K(b):K]$ then $v(b-c)<v(a-b)$. 
\end{itemize}
which precisely says that $(b,v(a-b))$ is a minimal pair with respect to $K$. In fact, if iii) holds and $c\in\overline{K}$ is such that $[K(c):K]<[K(b):K]$ then $v(b-c)=v(b-a+a-c)=v(a-c)<v(a-b)$. Similarly, iii') implies iii). Note also that iii') is equivalent to 
$$v(a-b)>\delta_K(b).$$
In particular, by the above Theorem, $\Gamma_{b}\subseteq\Gamma_{a}$, $k_{b}\subseteq k_{a}$ and $[K(b):K]\mid [K(a):K]$. 

Finally, note also that  $\delta_K(b)<\delta_K(a)$.
\end{Rems} 

\section{Stacked pseudo-convergent sequences of  $\oQp$}\label{stacked sect}

Let $\PP\subset\Z$ be the set of prime numbers and let $p\in\PP$ be a fixed prime. We let $\Z_{(p)}$ be the localization of $\Z$ at the prime ideal $p\Z$, $\Z_p$ the ring of $p$-adic integers and $\Q_p$ its field of fractions, the field of $p$-adic numbers. If $v_p$ denotes the usual $p$-adic valuation, then $\Z_p$ ($\Q_p$, respectively) is the completion of $\Z$ ($\Q$, respectively) with respect to the $p$-adic valuation. We denote by $\oQp$ a fixed algebraic closure of $\Q_p$ and still denote the unique extension of $v_p$ to $\oQp$ by $v_p$. Note that $\oQp$ is a rank one non-discrete valued field with valuation domain denoted by $\oZp$, the integral closure of $\Z_p$ in $\oQp$. We will use the well-known fact that $\Q_p$ has only finitely many extensions of a given degree  (see for example \cite[Corollary 2, Chapter V, p. 202]{Nark}).

Finally, we let $\C_p$ be the completion of $\oQp$ with respect to the $p$-adic valuation and by $\O_p$ the completion of $\oZp$; $v_p$ still denotes the unique extension of $v_p$ to $\C_p$. For $\alpha\in\oQp\setminus\Q_p$, for short we set $\delta_{\Q_p}(\alpha)=\delta(\alpha)$ and $\omega_{\Q_p}(\alpha)=\omega(\alpha)$. For $\alpha\in\C_p$, we denote by $e_{\alpha}$ ($f_{\alpha}$, respectively) the ramification index (the residue field degree, respectively) of $\Q_p(\alpha)$ over $\Q_p$. Clearly, if $\alpha\in\oQp$ then $e_{\alpha}\cdot f_{\alpha}<\infty$; we show that the converse holds in Remark \ref{eafaoQp}. Note that each element of $\C_p\setminus\oQp$ is transcendental over $\Q_p$; we call such elements simply transcendental. For a transcendental element $\alpha\in\C_p$, even if $e_{\alpha}\cdot f_{\alpha}=\infty$, we will show in Theorem \ref{prescribed residue field and value group} that either one of $e_{\alpha}$ or $f_{\alpha}$ can be finite.  



\subsection{Residually algebraic torsion extensions of $\Z_p$}\label{stackedseq}

In this section we describe residually algebraic torsion extensions of $\Z_p$ to $\Q_p(X)$ by means of a suitable class of pseudo-convergent sequences of transcendental type contained in $\oQp$, called stacked sequence, which we now introduce. This definition is a generalization of \cite[p. 135]{APZclosed}\footnote{The notion of distinguished sequence was introduced in \cite{APZclosed}. We cannot borrow that term here for our sequences for the following reason: by Lemma \ref{distinguishedpcv}, a stacked sequence is pseudo-convergent, and distinguished pseudo-convergent sequences have already been defined by P. Ribenboim on p. 474 of   \emph{Corps maximaux et complets par des valuations de Krull}. Math. Z. 69 (1958), 466–479, to denote pseudo-convergent sequences of a valued field whose breadth ideal is a non-maximal prime ideal.}.

\begin{Def}\label{stacked}
Let $E=\{s_n\}_{n\geq0}\subset\oQp$ be a sequence with $s_0\in\Q_p$. For every $n\geq0$, we consider the following properties:
\begin{itemize}
\item[i)] $[\Q_p(s_n):\Q_p]<[\Q_p(s_{n+1}):\Q_p]$.
\item[ii)] for every $c\in\oQp$ such that $[\Q_p(c):\Q_p]<[\Q_p(s_{n+1}):\Q_p]$, $v(s_{n+1}-c)\leq v(s_{n+1}-s_n)$.
\item[iii)] for every $c\in\oQp$ such that $[\Q_p(c):\Q_p]<[\Q_p(s_{n}):\Q_p]$, $v(s_{n}-c)<v(s_{n+1}-s_n)$.
\end{itemize}
We say that $E$ is \emph{unbounded} if i) holds for every $n$, \emph{stacked} if i) and iii) hold for every $n$, and \emph{strongly stacked} if i), ii), iii) hold for every $n$. Equivalently, $E$ is stacked if i) holds and $(s_n,\delta_n=v(s_{n+1}-s_n))$ is a minimal pair for every $n\geq0$ and $E$ is strongly stacked if $(s_{n},s_{n+1})$ is distinguished for every $n\geq0$.
\end{Def}

\begin{Rem}\label{pcvdist}
Let $E=\{s_n\}_{n\in\N}\subset\oQp$ be a stacked sequence. Note that the sequence $\{v(s_{n+1}-s_{n})=\delta_{n}\}_{n\in\N}$ is strictly increasing, since $[\Q_p(s_{n-1}):\Q_p]<[\Q_p(s_n):\Q_p]$ and $(s_n,\delta_n)$ is a minimal pair.   
In the original definition of distiguished sequence $E$ in \cite{APZclosed}, the sequence $\{\delta_n\}_{n\in\N}$  is   unbounded, thus, in this case   $E$ is a Cauchy sequence. In our setting we are not imposing that restriction; we show in Lemma \ref{distinguishedpcv} below that a stacked sequence is a pseudo-convergent sequence of transcendental type of $\oQp$.

The motivation for the terminology of these kind of sequences, is due to the following fact. For each $n\in\N$, for short we set $\Gamma_n=\Gamma_{s_n}$ and $k_n=k_{s_n}$ (i.e., the value group and the residue field  of the valuation domain $O_{s_n}$ of $\Q_p(s_n)$, respectively).  By Remarks \ref{properties distinguised pair}, $v(s_{n+1}-s_n)>\delta(s_n)$. Hence, by Theorem \ref{propertydistinguished}, we have $\Gamma_{n}\subseteq \Gamma_{n+1}$ and $k_{n}\subseteq k_{n+1}   $. For each $n\in\N$, we set $e_n=e(\Q_p(s_n)\mid\Q_p)$ and $f_n=f(\Q_p(s_n)\mid\Q_p)$, the ramification index and the residue field degree of $O_{s_n}$ over $\Z_p$, respectively; we remark that $[\Q_p(s_n):\Q_p]=e_nf_n=d_n$ for each $n\in\N$, and since $\{d_n\}_{n\in\N}$ is unbounded by assumption, either $\{e_n\}_{n\in\N}$ is unbounded or $\{f_n\}_{n\in\N}$ is unbounded. Since $e_n\mid e_{n+1}$ for each $n\in\N$, $\{e_n\}_{n\in\N}$ is bounded if and only if $e_n=e$ for all $n\in\N$ sufficiently large. Similarly for $\{f_n\}_{n\in\N}$. 

By Remarks \ref{properties distinguised pair}, condition ii) is equivalent to $\delta_n=v(s_n-s_{n+1})=\delta(s_{n+1})$ (note that in general the inequality $\delta_n\leq\delta(s_{n+1})$ holds). In other words, among all the elements $c\in\oQp$ such that $[\Q_p(s_n):\Q_p]\leq[\Q_p(c):\Q_p]<[\Q_p(s_{n+1}):\Q_p]$, $s_n$ is one of those which is closest to $s_{n+1}$. 
\end{Rem}

Let $E=\{t_n\}_{n\in\N}\subset\oQp$ be a pseudo-convergent sequence. If $\{[\Q_p(t_n):\Q_p]\mid n\in\N\}$ is bounded, then $E$ is contained in a finite extension $K$ of $\Q_p$ and hence $E$ is Cauchy and therefore converges to an element $\alpha\in K$. In particular, if $E$ is  of transcendental type, then the set $\{[\Q_p(t_n):\Q_p]\}_{n\in\N}$ is necessarily unbounded.  Stacked sequences are of this kind, as the next lemma shows.

\begin{Lem}\label{distinguishedpcv}
Let $E\subset\oQp$ be a stacked sequence. Then $E$ is a pseudo-convergent sequence of transcendental type.
\end{Lem}
\begin{proof} Let $E=\{s_n\}_{n\in\N}$ and set $\delta_n=v(s_{n+1}-s_n)$, for each $n\in\N$. We have already observed in Remark \ref{pcvdist} that $\{\delta_n\}_{n\in\N}$ is a strictly increasing sequence. Moreover, for every $m>n$, we have $v(s_n-s_m)>v(s_n-s_{n-1})$. In particular, $v(s_{n-1}-s_m)=v(s_{n-1}-s_n)$ for every $m\geq n$. Let now $n<m<k$. Then
$$v(s_n-s_m)=v(s_n-s_{n+1})<v(s_m-s_{m+1})=v(s_m-s_k)$$
which shows that $E$ is a pseudo-convergent sequence.

We prove now that $E$ is of transcendental type. Let $\alpha\in\oQp$. Then there exists $n\in\N$ such that $[\Q_p(\alpha):\Q_p]<[\Q_p(s_n):\Q_p]$. Since $(s_n,\delta_n)$ is a minimal pair, $v(s_n-\alpha)<\delta_{n}$, so, in particular, $\alpha$ cannot be a pseudo-limit of $E$. This shows that $E$ has no pseudo-limits in $\oQp$, thus $E$ is of transcendental type.
\end{proof}

Let $E=\{s_n\}_{n\in\N}\subset\oQp$ be a stacked sequence. In particular, by Lemma \ref{distinguishedpcv}, the sequence $\{\delta_n=v(s_{n+1}-s_n)\}_{n\in\N}$ is the gauge of the pseudo-convergent sequence $E$. Moreover, by the same Lemma, if $E$ is Cauchy, then $E$ converges to a transcendental element $\alpha\in\C_p$. 

The next proposition shows that any residually algebraic torsion extension of $\Z_p$ to $\Q_p(X)$ is obtained by means of a pseudo-convergent sequence of transcendental type of $\oQp$. We recall that if $E\subset\oQp$ is a pseudo-convergent sequence of transcendental type, then $\oZp_E$, the associated valuation domain of $\oQp(X)$, is   an immediate extension of $\oZp$ and conversely every immediate extension of $\oZp$ to $\oQp(X)$ can be realized in this way (see for example \cite{Kap, PS2}). If $\Z_{p,E}=\oZp_E\cap\Q_p(X)$, then $\Z_{p,E}$ is a residually algebraic torsion extension of $\Z_p$ to $\Q_p(X)$.

\begin{Prop}\label{ratpcvtransc}
Let $W$ be a residually algebraic torsion extension of $\Z_p$ to $\Q_p(X)$. Then there exists a pseudo-convergent sequence $E\subset\oQp$ of transcendental type such that 
$$W=\Z_{p,E}=\{\phi\in\Q_p(X)\mid \phi(s_n)\in \oZp,\text{ for all sufficiently large }n\in\N\}.$$
\end{Prop}
\begin{proof}
Let $\overline{W}$ be an extension of $W$ to $\oQp(X)$. Then $\overline{W}$ is an immediate extension of $\oZp$ to $\oQp(X)$ (and, in particular, is a residually algebraic torsion extension of $\oZp$). By \cite[Theorems 1 \& 2]{Kap} or \cite[Theorem 6.2, a)]{PS2}, there exists a pseudo-convergent sequence $E\subset\oQp$ of transcendental type such that $\overline{W}=\oZp_E$. The claim follows contracting down to $\Q_p(X)$.
\end{proof}


Clearly, not every pseudo-convergent sequence of transcendental type in $\oQp$ is stacked. However, the next theorem is the converse of Lemma \ref{distinguishedpcv}: it shows that any  pseudo-convergent sequence of transcendental type is equivalent to a strongly stacked  sequence. In particular, every stacked sequence is equivalent to a strongly stacked sequence.  Moreover, given a valuation domain $\Z_{p,E}$ of $\Q_p(X)$ associated to a pseudo-convergent sequence $E\subset\oQp$ of transcendental type,   without loss of generality we may also assume that $E$ is strongly stacked.

By \cite[Proposition 2.2]{APZclosed}, every transcendental element $t\in\C_p$ is the limit of a strongly stacked sequence $E$ of $\oQp$. The next theorem is the analogous of that result for residually algebraic extensions $W$ of $\Z_p$ to $\Q_p(X)$: for such a valuation $W$, there exists a strongly stacked sequence $E\subset\oQp$ such that $W=\Z_{p,E}$; it is not difficult to show that, for a transcendental element $t\in\C_p$, the valuation domain $\Z_{p,t}=\{\phi\in\Q_p(X)\mid\phi(t)\in \O_p\}$ is a residually algebraic torsion extension of $\Z_p$. 

\begin{Thm}\label{Epcvtranscdistinguished}
Let $E\subset\oQp$ be a pseudo-convergent sequence of transcendental type. Then there exists a strongly stacked sequence $E'\subset\oQp$ which is equivalent to $E$. In particular, given a residually algebraic torsion extension $W$ of $\Z_p$ to $\Q_p(X)$, there exists a strongly stacked sequence $E'\subset\oQp$ such that $W=\Z_{p,E'}$.
\end{Thm}

\begin{proof}
Let $E=\{t_n\}_{n\in\N}$ and let $\overline v_E$ be a valuation associated to $\oZp_{E}\subset\oQp(X)$. 

First, we consider the following subset of $\Gamma_{v_E}\subseteq\Q$:
$$M_E(X,\Q_p)=\{v_E(X-s)\mid s\in\Q_p\}.$$
If $M_E(X,\Q_p)$ is not bounded, then there exists a sequence $\{s_n\}_{n\in\N}\subset\Q_p$ such that $v_E(X-s_n)$ tends to $\infty$. Necessarily, the sequence $\{s_n\}_{n\in\N}$ is  Cauchy and so converges to an element $s$ of $\Q_p$. Now, for every $n$, $v_E(X-s_n)=\overline v_E(X-s_n)=v(t_m-s_n)$ for all $m$ sufficiently large since $E$ is of transcendental type (see \S \ref{pcv}). Hence, $E$ would be a Cauchy sequence equivalent to $\{s_n\}_{n\in\N}$ and $E$ would converge to $s$, too, which is not possible. Let then $\delta_0=\sup M_E(X,\Q_p)\in\R$. We claim that $\delta_0\in M_E(X,\Q_p)$, that is, $\delta_0$ is a maximum. Suppose otherwise: there exists a sequence $\{r_k\}_{k\in\N}\subset\Q_p$ such that $v_E(X-r_k)\nearrow\delta_0$. Then $\{r_k\}_{k\in\N}\subset\Q_p$ would be a pseudo-convergent sequence which is not Cauchy, which is not possible since $\Q_p$ is a complete valued field. Hence, there exists $s_0\in\Q_p$ such that $v_E(X-s_0)=\delta_0$.

For $n>0$, we now choose $s_n\in\oQp$ so that $(s_{n-1},s_n)$ is distinguished. Let $B_n$ be the subset of $\alpha$'s in $\oQp$ satisfying the following properties:
\begin{itemize}
\item[i)] $[\Q_p(\alpha):\Q_p]>[\Q_p(s_{n-1}):\Q_p]$.
\item[ii)] $\overline v_E(X-\alpha)>\overline v_E(X-s_{n-1})$.
\item[iii)] the positive integer $[\Q_p(\alpha):\Q_p]-[\Q_p(s_{n-1}):\Q_p]$  is minimal.
\end{itemize}
Note that since $\N$ is well-ordered, condition iii) can be satisfied (that is, among the $\alpha\in\oQp$ satisfying i) and ii), we can find one which also satisfies iii) ). Since $\overline v_E(X-s_{n-1})=v(t_m-s_{n-1})$ for all $m$ sufficiently large, for such $m$'s we also have $\overline v_E(X-s_{n-1})<\overline v_E(X-t_m)$. Moreover, without loss of generality, we may also assume that $[\Q_p(t_m):\Q_p]>[\Q_p(s_{n-1}):\Q_p]$ since $\{[\Q_p(t_m):\Q_p]\}_{m\in\N}$ is unbounded.  This shows that the set $B_n$ is non-empty. Let $M_E(X,B_n)=\{\overline v_E(X-\alpha)\mid\alpha\in B_n\}$, which is a subset of $\Q$. Let $\delta_n=\sup M_E(X,B_n)$. Since each element of $B_n$ has the same degree over $\Q_p$, it follows that $B_n$ is contained in a finite extension $K$ of $\Q_p$. In particular, it follows as above that $M_E(X,B_n)$ is upper bounded. Let $\delta_n=\sup M_E(X,B_n)\in\R$. Next, we show that $M_E(X,B_n)$ contains its upper bound (which is, in particular, a rational number). Suppose otherwise: then there exists a sequence $\{\alpha_k\}_{k\in\N}\subset B_n$ such that $\overline v_E(X-\alpha_k)\nearrow\delta_n$. In particular, $\{\alpha_k\}_{k\in\N}$ would be a pseudo-convergent sequence of a finite extension of $\Q_p$ which is not Cauchy, which is impossible. Let $s_n\in B_n$ be such that $\overline v_E(X-s_n)=\delta_n$. Note that $v_p(s_n-s_{n-1})=\overline v_E(s_n-X+X-s_{n-1})=\overline v_E(X-s_{n-1})=\delta_{n-1}$. 

We now show that $(s_{n-1},s_n)$ is distinguished.  Clearly, $[\Q_p(s_{n-1}):\Q_p]<[\Q_p(s_n):\Q_p]$. 

Let $c\in\oQp$ be such that $[\Q_p(c):\Q_p]<[\Q_p(s_n):\Q_p]$. 

If $[\Q_p(c):\Q_p]>[\Q_p(s_{n-1}):\Q_p]$, then by the minimality of the degree of $s_n$, we have $\overline v_E(X-c)\leq \overline v_E(X-s_{n-1})=\delta_{n-1}$, so
$$v_p(s_n-c)=\overline v_E(s_n-X+X-c)=\overline v_E(X-c)\leq \delta_{n-1}=v_p(s_n-s_{n-1})$$
Suppose now that $[\Q_p(c):\Q_p]=[\Q_p(s_{n-1}):\Q_p]$.

If $\overline v_E(X-c)\leq \overline v_E(X-s_{n-2})$, then $\overline v_E(X-c)<\delta_{n-1}$.

If $\overline v_E(X-c)> \overline v_E(X-s_{n-2})$ then $c\in B_{n-1}$, so $ \overline v_E(X-c)\leq \delta_{n-1}=\overline v_E(X-s_{n-1})$.

In either case,
$$v_p(s_n-c)=\overline v_E(s_n-X+X-c)=\overline v_E(X-c)\leq\delta_{n-1}=v_p(s_n-s_{n-1}).$$
Note that in particular, for $n=1$ we have that $(s_0,s_1)$ is distinguished, since condition iii) of Definition \ref{stacked} is empty, since $s_0\in\Q_p$.

Suppose now that $n\geq 2$ and assume by induction that $(s_{n-2},s_{n-1})$ is distinguished. Let $c\in\oQp$ be such that $[\Q_p(c):\Q_p]<[\Q_p(s_{n-1}):\Q_p]$. Since $(s_{n-2},s_{n-1})$ is distinguished, we have $v_p(s_{n-1}-c)\leq v_p(s_{n-1}-s_{n-2})=\delta_{n-2}<\delta_{n-1}$. Hence, 
$$v_p(s_n-c)=v_p(s_n-s_{n-1}+s_{n-1}-c)=v_p(s_{n-1}-c)<v_p(s_n-s_{n-1}).$$

We now show that $E'=\{s_n\}_{n\in\N}$ is equivalent to $E=\{t_n\}_{n\in\N}$. Let $\{\lambda_n\}_{n\in\N}$ and  $\{\delta_n\}_{n\in\N}$ be the gauges of $E$ and $E'$, respectively. We need to show that, for each $k\in\N$, there exists $n\in \N$ such that $\lambda_k\leq\delta_{n}$. Since $E'$ is unbounded, there exists $n\in\N$ such that $[\Q_p(t_k):\Q_p]<[\Q_p(s_n):\Q_p]$. Since $(s_n,\delta_n)$ is a minimal pair,  we have $v_p(s_n-t_k)<\delta_n$, so that 
\begin{equation}\label{lambdak}
\lambda_k=\overline v_E(X-t_k)=\overline v_E(X-s_n+s_n-t_k)<\overline v_E(X-s_{n})=\delta_{n}
\end{equation}

Conversely, let $n\in\N$. We need to show that there exists $k\in \N$ such that $\delta_n\leq\lambda_k$. For all $m$ sufficiently large we have
$$\overline v_E(X-s_n)=v_p(t_m-s_n)=\overline v_E(t_m-X+X-s_n)$$
and since $n$ is fixed and $\overline v_E(t_m-X)=\lambda_m$ is strictly increasing, it follows that $\overline v_E(t_m-X)=\lambda_m>\overline v_E(X-s_n)$
for all such $m$'s. 

Hence,  $\Br(E)=\Br(E')$.

Finally, we need to show that, if $k\in\N$, then there exist $n_0,m_0\in\N$ such that for each $n\geq n_0$ and $m\geq m_0$, we have $v_p(t_n-s_m)>\lambda_k$. Let $n_0$ be the smallest integer such that $[\Q_p(t_k):\Q_p]<[\Q_p(s_{n_0}):\Q_p]$. As in \eqref{lambdak} above, $\lambda_k<v_E(X-s_{n_0})=\delta_{n_0}$.  Let now $m>k$ and $n\geq n_0$. Then,
$$v(t_m-s_n)=\overline v_E(t_m-X+X-s_n)>\lambda_k$$
since $\overline v_E(t_m-X)=\lambda_m>\lambda_k$ and $\overline v_E(X-s_n)\geq \overline v_E(X-s_{n_0})=\delta_{n_0}>\lambda_k$. Hence, $E$ and $E'$ are equivalent.

By \cite[Proposition 5.3]{PS2}, $\oZp_{E}=\oZp_{E'}$, so, in particular, $\Z_{p,E}=\Z_{p,E'}$. The final claim follows by Proposition \ref{ratpcvtransc}.
\end{proof}

The following proposition describes the value group and the residue field of a residually algebraic torsion extension $W$ of $\Z_p$ to $\Q_p(X)$. By Theorem \ref{Epcvtranscdistinguished}, $W$  is equal to $\Z_{p,E}$, for some strongly stacked sequence $E\subset\oQp$.  We keep the notation of   Remark \ref{pcvdist}.

\begin{Prop}\label{residuevaluerat}
Let $E=\{s_n\}_{n\in\N}\subset\oQp$ be a stacked sequence and $W=\Z_{p,E}$.  
Then we have 
$$\bigcup_{n\in\N} \Gamma_n=\Gamma_w,\;\;\bigcup_{n\in\N} k_n=k_w.$$
\end{Prop}
\begin{proof}
Let  $w=v_E$ be the valuation associated  to $\Z_{p,E}$ and $\overline v_E$ the valuation associated to $\oZp_E$.

Since $E$ is of transcendental type, for each $f\in \Q_p[X]$, $v_E(f)=v(f(s_n))$ for all $n$ sufficiently large (see \S \ref{pcv}). It follows that for each $\phi\in\Q_p(X)$, $\phi=f/g$, for some $f,g\in\Q_p[X]$, $v_E(\phi)=v_E(f)-v_E(g)$ is in $\Gamma_n$ for all $n$ sufficiently large. Hence, $\Gamma_w\subseteq \bigcup_n \Gamma_n$. Conversely, let $n\in\N$ and $f\in\Q_p[X]$ be of degree smaller than $[\Q_p(s_n):\Q_p]$. Then, each root $\alpha_i$ of $f(X)$ in $\oQp$ has degree  smaller than $[\Q_p(s_n):\Q_p]$ and so, since $(s_n,\delta_n)$ is a minimal pair, we have
\begin{equation}\label{smaller degree smaller value}
v_p(s_n-\alpha_i)<\delta_n
\end{equation}
which implies that
\begin{equation}\label{vEX-alpha}
\overline v_E(X-\alpha_i)=\overline v_E(X-s_n+s_n-\alpha_i)=v_p(s_n-\alpha_i)
\end{equation}
and so
\begin{equation}\label{vEfvfsn}
v_E(f(X))=\sum_i \overline v_E(X-\alpha_i)=\sum_i v_p(s_n-\alpha_i)=v_p(f(s_n))
\end{equation} 
which shows that $\Gamma_n\subseteq \Gamma_w$. Note that $\overline v_E(X-\alpha_i)=v(s_m-\alpha_i)$ for each $m\geq n$ and so $v_E(f(X))=v(f(s_m))$ for each $m\geq n$.

Let now $n\in\N$ and $\overline{c}=\overline{f(s_n)}\in k_n$, for some $f(s_n)\in O_n^*$ where $f\in \Q_p[X]$ has degree strictly smaller than $[\Q_p(s_n):\Q_p]$. In particular, $\overline{c}\not=0$. As in \eqref{vEfvfsn}, $v_E(f(X))=v(f(s_m))=0$ for each $m\geq n$. Let $\alpha_i$ be a root of $f(X)$ in $\oQp$. Then by \eqref{vEX-alpha}, $\overline v_E(X-\alpha_i)=v_p(s_n-\alpha_i)=v_p(d_i)$ for some $d_i\in\oQp$. Then
$$\overline v_E\left(\frac{(X-\alpha_i)/d_i}{(s_n-\alpha_i)/d_i}-1\right)=\overline v_E\left(\frac{X-s_n}{s_n-\alpha_i}\right)=\delta_n-v_p(s_n-\alpha_i)>0$$
where the last inequality holds by \eqref{smaller degree smaller value}. Therefore, $(X-\alpha_i)/d_i$ and $(s_n-\alpha_i)/d_i$ coincide over the residue field of $W$. In particular,
\begin{equation}\label{f(X)f(sn)}
\frac{f(X)}{f(s_n)}=\prod_i\frac{(X-\alpha_i)}{(s_n-\alpha_i)}=\prod_i\frac{(X-\alpha_i)/d_i}{(s_n-\alpha_i)/d_i}
\end{equation}
and since each factor of the last product has residue $\overline 1$ in $W$, it follows that $f(X)$ and $f(s_n)$ coincide over the  residue field of $W$ (which contains both $f(X)$ and $f(s_n)$). Since $f\in\Z_{p,E}=W$, this shows that $k_n$ is contained in the residue field $k_w$ of $W$.

Conversely, let $\phi=f/g\in W\subset\Q_p(X)$, for some $f,g\in\Q_p[X]$. Let $\alpha_i,\beta_j$ be the roots in $\oQp$ of $f$ and $g$, respectively. There exists $n\in\N$ such that $[\Q_p(\alpha_i):\Q_p]<[\Q_p(s_n):\Q_p]$ and $[\Q_p(\beta_j):\Q_p]<[\Q_p(s_n):\Q_p]$ for all $i$ and $j$. Hence, as in \eqref{vEX-alpha} we have 
$$\overline v_E(X-\alpha_i)=v_p(s_n-\alpha_i),\;\;\overline v_E(X-\beta_j)=v(s_n-\beta_j),\forall i,j$$
which again as in \eqref{vEfvfsn} shows that
$$v_E(\phi(X))=v(\phi(s_n))$$
Moreover, this last equation holds if we replace $s_n$ by $s_m$, for all $m\geq n$. If $v_E(\phi(X))=0$, then as in \eqref{f(X)f(sn)} one can show that $\phi(X)$ and $\phi(s_n)$ coincide over the residue field of $W$, so that $k_w\subseteq k_n$.
\end{proof}
The following corollary gives a further characterization of the residue field and the value group of a residually algebraic torsion extension $W$ of $\Z_p$: either the residue field of $W$ is an infinite algebraic extension of $\F_p$, or the value group $\Gamma_w$ is non-discrete.
\begin{Cor}\label{eof infinite}
Let $W$ be a residually algebraic torsion extension of $\Z_p$ to $\Q_p(X)$ and let $e=e(W|\Z_p)$ and $f=f(W|\Z_p)$ be the ramification index and the residue field degree, respectively.  Then $e\cdot f=\infty$.
\end{Cor}
\begin{proof}
By Theorem \ref{Epcvtranscdistinguished}, there exists a stacked sequence $E=\{s_n\}_{n\in\N}\subset\oQp$ such that $W=\Z_{p,E}$. By Proposition \ref{residuevaluerat}, $\Gamma_w=\bigcup_n \Gamma_n$ and $k_w=\bigcup_n k_n$. Remark \ref{pcvdist} shows that either the sequence $\{e_n=[\Gamma_n:\Z]\}_{n\in\N}$ or $\{f_n=[k_n:\F_p]\}_{n\in\N}$ is unbounded, therefore, either $e=e(W|\Z_p)$ or $f=f(W|\Z_p)$ is infinite.
\end{proof}

The following proposition is analogous to \cite[Proposition 2.3]{APZclosed}. It shows that the sequence of ramification indexes,   residue field degrees and gauges attached to a residually algebraic torsion extension $W$ of $\Z_p$ do not depend on the strongly stacked  sequence $E\subset \oQp$ such that $W=\Z_{p,E}$ (Theorem \ref{Epcvtranscdistinguished}).

\begin{Prop}
Let $W\subset\Q_p(X)$ be a residually algebraic torsion extension of $\Z_p$. Let $E=\{s_n\}_{n\in\N},E'=\{t_n\}_{n\in\N}\subset\oQp$ be strongly  stacked sequences with gauges $\{\delta_n\}_{n\in\N},\{\delta'_n\}_{n\in\N}$, respectively, such that $W=\Z_{p,E}=\Z_{p,E'}$. Then, for each $n\in\N$ we have:
\begin{itemize}
\item[i)] $[\Q_p(s_n):\Q_p]=[\Q_p(t_n):\Q_p]$ and $\delta_n=\delta_n'$.
\item[ii)] $e_{s_n}=e_{t_n}$ and $f_{s_n}=f_{t_n}$.
\end{itemize}
\end{Prop}
\begin{proof}
Without loss of generality, we may assume that in $\oQp(X)$ we have $\oZp_E=\oZp_{E'}$; we let $\overline W=\oZp_E=\oZp_{E'}$ and we denote by $w$ a valuation associated to $\overline W$. 


We prove i). We have $s_0,t_0\in\Q_p$. There exists $n\in\N$, $n\geq1$, such that $w(X-s_{n-1})\leq w(X-t_0)<w(X-s_n)$, otherwise $t_0$ would be a pseudo-limit of $E$, which is not possible. In particular, $v_p(s_n-t_0)=w(s_n-X+X-t_0)=w(X-t_0)\geq w(X-s_{n-1})=\delta_{n-1}$. If $n>1$ we have $[\Q_p(t_0):\Q_p]<[\Q_p(s_{n-1}):\Q_p]$ so by iii) of Definition \ref{stacked} we have that $v_p(s_n-t_0)=v_p(s_{n-1}-t_0)<v_p(s_n-s_{n-1})=\delta_{n-1}$ which   is impossible. Hence, $n=1$, so $v_p(s_1-t_0)=w(X-t_0)\geq w(X-s_0)$. Reversing the roles of $s_0,t_0$, we get the other inequality, so $w(X-s_0)=w(X-t_0)=\delta_0=\delta_0'$.

Let $n\in\N$ and suppose that for each $m\leq n$ we have $[\Q_p(s_m):\Q_p]=[\Q_p(t_m):\Q_p]$ and $\delta_m=\delta_m'$.

Since $[\Q_p(s_{n}):\Q_p]=[\Q_p(t_{n}):\Q_p]<[\Q_p(t_{n+1}):\Q_p]$, by ii) of Definition \ref{stacked} we have  $v_p(t_{n+1}-s_n)\leq v_p(t_{n+1}-t_n)=\delta_n'=\delta_n$. Now, 
$$v_p(t_{n+1}-s_n)=v_p(t_{n+1}-t_n+t_n-s_n)\geq\delta_n$$
since $v_p(t_n-s_n)=w(t_n-X+X-s_n)\geq \delta_n=\delta_n'$. This implies that $v_p(t_{n+1}-s_n)=\delta_n$ and so $(s_n,t_{n+1})$  is distinguished. Moreover, we have
$$v_p(t_{n+1}-s_{n+1})=w(t_{n+1}-X+X-s_{n+1})>\delta_n=\delta_n'=v_p(t_{n+1}-s_n).$$
Now, if $[\Q_p(s_{n+1}):\Q_p]< [\Q_p(t_{n+1}):\Q_p]$, then since $(s_n,t_{n+1})$ is distinguished, we would have $v_p(s_{n+1}-t_{n+1})\leq v_p(t_{n+1}-s_n)$, which is impossible. Hence, $[\Q_p(s_{n+1}):\Q_p]\geq [\Q_p(t_{n+1}):\Q_p]$. The other inequality is proved in a symmetrical way, so $[\Q_p(s_{n+1}):\Q_p]= [\Q_p(t_{n+1}):\Q_p]$.

Suppose now that $w(X-s_{n+1})<w(X-t_{n+1})$. Then $v_p(s_{n+2}-t_{n+1})=w(s_{n+2}-X+X-t_{n+1})>w(X-s_{n+1})=v_p(s_{n+2}-s_{n+1})$ which is not possible since $(s_{n+1},s_{n+2})$ is distinguished. Hence, $w(X-s_{n+1})\geq w(X-t_{n+1})$. The other inequality is proved similarly, so $\delta_{n+1}=\delta_{n+1}'$ as claimed.

We prove now ii). For each $n\in\N$, let $\Gamma_n,\Gamma'_n$ and $k_n,k_n'$ be the value groups and residue fields, respectively, of $\Q_p(s_n)$ and $\Q_p(t_n)$. Let $e_n=e_{s_n}$, $e_n'=e_{t_n}$, $f_n=f_{s_n}$, $f'_n=f_{t_n}$.

Clearly, $e_0=e_0'$ and $f_0=f_0'$, since $s_0,t_0\in\Q_p$. 

Let $n\geq 1$. If $f\in \Q_p[X]$ has  degree strictly smaller than $[\Q_p(s_n):\Q_p]=[\Q_p(t_n):\Q_p]$, then by \eqref{vEfvfsn} we have $w(f(X))=v_p(f(s_n))$ and also $w(f(X))=v_p(f(t_n))$, so $v_p(f(s_n))=v_p(f(t_n))$. This proves that $\Gamma_n=\Gamma_n'$ and so $e_n=e_n'$.

Suppose now that $f\in \Q_p[X]$ of degree strictly smaller than $[\Q_p(s_n):\Q_p]=[\Q_p(t_n):\Q_p]$ is such that $v_p(f(s_n))=v_p(f(t_n))=0$. In particular,  $w(f(X))=0$ by \eqref{vEfvfsn}. By \eqref{f(X)f(sn)}  and the analogous equation where $s_n$ is   replaced by $t_n$, we get that $f(s_n),f(t_n)$ have the same residue as $f(X)$, so in particular, $k_n=k_n'$. Therefore, $f_n=f_n'$.
\end{proof}

\subsection{Residually algebraic extensions of $\Z_p$ which are DVRs}

In this section we characterize DVRs of $\Q_p(X)$ extending $\Z_p$ such that the residue field extension is algebraic, necessarily of infinite degree by Corollary \ref{eof infinite}; this fact has already been noted in a different way in \cite[p. 4217]{PerTransc}. We will see in \S \ref{RatZpQX} that there is no such restriction on the residue field degree for DVRs of $\Q(X)$ which are residually algebraic extensions of $\Z_{(p)}$ (see Corollary \ref{DVRQ(X)rat}).

Given $\alpha\in\C_p$, we denote by $\O_{p,\alpha}$ the unique valuation domain of $\Q_p(\alpha)$ lying over $\Z_p$ (i.e., $\O_{p,\alpha}=\O_p\cap\Q_p(\alpha)$). We also set
$$\Z_{p,\alpha}=\{\phi\in\Q_p(X)\mid \phi(\alpha)\in \O_p\},$$
which is a valuation domain of $\Q_p(X)$ and coincides with the previous definition if $\alpha\in\oQp$. 

\begin{Prop}\label{ef transcendental elements Cp}
Let $\alpha\in\C_p$ be a transcendental element. Then there exists a Cauchy stacked sequence $E\subseteq\oQp$ converging to $\alpha$. Moreover, the valued fields $(\Q_p(X),\Z_{p,\alpha})$ and $(\Q_p(\alpha),\O_{p,\alpha})$ are isomorphic.  In particular, the ramification index $e(\Z_{p,\alpha}|\Z_p)$ is equal to $e_{\alpha}$, the residue field degree $f(\Z_{p,\alpha}|\Z_p)$ is equal to $f_{\alpha}$ and $e_{\alpha}\cdot f_{\alpha}=\infty$. 
\end{Prop}
Note that the last condition implies that either $e_{\alpha}$ or $f_{\alpha}$ is infinite. It can happen that exactly one of these two quantities is finite (see Theorem \ref{prescribed residue field and value group}).
\begin{proof}
The proof of the first claim follows also by \cite[Proposition 2.2]{APZclosed}, but we give here a different proof based on the previous results. 


By Theorem \ref{Epcvtranscdistinguished}, there exists a stacked sequence $E\subset\oQp$ such that $\Z_{p,\alpha}=\Z_{p,E}$. Since the valuation domains $\oZp_{E},\oZp_{\alpha}\subset\oQp(X)$ contracts down to $\Q_p(X)$ to the same valuation domain, there exists $\sigma\in\Gal(\oQp/\Q_p)$ such that $\sigma(\oZp_{\alpha})=\oZp_{\sigma(\alpha)}=\oZp_{E}$. By \cite[Proposition 5.3]{PS2}, $E$ is then a Cauchy sequence converging to $\sigma(\alpha)$. Since $\Z_{p,\alpha}=\Z_{p,\sigma(\alpha)}$, without loss of generality we may assume that $E$ converges to $\alpha$. 

Since $\alpha$ is transcendental over $\Q_p$, the evaluation homomorphism $\text{ev}_{\alpha}:\Q_p(X)\to\Q_p(\alpha)$, $\phi(X)\mapsto\phi(\alpha)$, is an isomorphism. It is easy to see that $\text{ev}_{\alpha}(\Z_{p,\alpha})=\O_{p,\alpha}$. Hence, $\Z_{p,\alpha}$ and $\O_{p,\alpha}$ have the same ramification indexes and residue field degrees over $\Z_p$. 

Finally, the last claim follows by Corollary \ref{eof infinite}.
\end{proof}
\begin{Rem}\label{eafaoQp}
By Proposition \ref{ef transcendental elements Cp}, we may conclude that in general,
\begin{equation*}
\text{for }\alpha\in\C_p, \text{ we have }e_{\alpha}\cdot f_{\alpha}<\infty \text{ if and only if }\alpha\in\oQp.
\end{equation*}
\end{Rem}

The next lemma may be well-known, but lacking a reference we give a short proof.
\begin{Lem}\label{ramification compositum}
Let $p\in\Z$ be a prime, $K_1,K_2$ finite extensions of $\Q_p$ and $L=K_1K_2$ the compositum. Let $e_1$ be the ramification index of $K_1$ over $\Q_p$ and $e$ the ramification index of $L$ over $K_2$. Then $e\leq e_1$.
\end{Lem}
\begin{proof}
If $K_1$ is a tame extension of $\Q_p$, then the ramification index of $L$ over $\Q_p$ is equal to lcm$\{e(K_1|\Q_p),e(K_2|\Q_p)\}$   (see for example \cite{ChabHal}), so $e$ divides $e_1$ and the claim is true. 

We give a self-contained proof which works in general. Let $L'$ be the normal closure of $L$ over $\Q_p$ and $I$ the inertia group of the maximal ideal $M_{L'}$  of $O_{L'}$ over $\Z_p$. Let $G_i$ be the Galois group $\Gal(L'|K_i)$, for $i=1,2$ and $G$ the Galois group $\Gal(L'|L)$. Since $L=K_1K_2$, we have $G=G_1\cap G_2$. The inertia group of $M_{L'}$ over $M_{K_1}$ is equal to $I\cap G_1$ and the inertia group of $M_{L'}$ over $M_L$ is equal to $I\cap G$. We have
$$e=\frac{e(L'|K_2)}{e(L'|L)}=\frac{\#(I\cap G_2)}{\#(I\cap G)},\;\;e_1=\frac{e(L'|\Q_p)}{e(L'|K_1)}=\frac{\#I}{\#(I\cap G_1)}$$
Note that $I\cap G=(I\cap G_1)\cap(I\cap G_2)$. Therefore, the claim follows by the following general fact for finite groups: given a finite group $G$ with two subgroups $H_1,H_2$, we have 
$$\frac{\#H_2}{\#(H_1\cap H_2)}=[H_2:H_1\cap H_2]\leq\frac{\#G}{\#H_1}=[G:H_1]$$
which follows immediately, since the map $h_2 H_1\cap H_2\mapsto h_2 H_1$ from the set $\{ h_2(H_1\cap H_2)\mid h_2\in H_2\}$ of left cosets of $H_1\cap H_2$ in $H_2$ to the set $\{gH_1\mid g\in G\}$ of left cosets of $H_1$ in $G$ is injective. 
\end{proof}

The following result is analogous to \cite[Theorem 2.5]{PerTransc}.

\begin{Thm}\label{DVRoverZp}
Let $W$ be a DVR of $\Q_p(X)$ which is a residually algebraic extension of $\Z_p$. Then there exists a transcendental element $\alpha\in \C_p$ such that $W=\Z_{p,\alpha}$. 
\end{Thm}
\begin{proof}
Note that, by Corollary \ref{eof infinite}, the residue field of $W$ is an infinite algebraic extension of $\mathbb F_p$.

By Theorem \ref{Epcvtranscdistinguished}, there exists a stacked sequence $E=\{s_n\}_{n\in\N}\subset\oQp$ such that $W=\Z_{p,E}$. By assumption, the ramification index $e(W|\Z_p)=e$ is finite. By Remark \ref{pcvdist} and Proposition \ref{residuevaluerat}, there exists $n_0\in\N$ such that $\Gamma_w=\Gamma_{n}=\Gamma_{n_0}$ for each $n\geq n_0$. Equivalently, $e_{n}=e_{n_0}=e$ for each $n\geq n_0$. Let $n\geq n_0$. Note that $\delta_n=v_p(s_{n+1}-s_n)\in \Gamma_{O_{K_n}}$, where $K_n=\Q_p(s_n,s_{n+1})$. Note that the ramification index of $\Q_p(s_i)$ over $\Q_p$ is equal to $e$, for $i=n,n+1$. By Lemma \ref{ramification compositum}, the ramification index of $K_n$ over $\Q_p$ is bounded by $e^2$. If $d=\prod_{i=1}^{e^2} i$, then $d\delta_n\in\Z$, for each $n\geq n_0$. This shows that the gauge $\{\delta_n\}_{n\in\N}$ of $E$ has bounded denominator, so $\delta_n\nearrow\infty$, thus $E$ is Cauchy and converges to a (unique) element $\alpha$ of $\C_p\setminus\oQp$, since $E$ is of transcendental type by Lemma \ref{distinguishedpcv}. In particular, $W=\Z_{p,\alpha}$.
\end{proof}

\begin{Rem}
We say that an element $\alpha\in\C_p$ has \emph{bounded ramification} if the extension $\Q_p(\alpha)\supseteq\Q_p$ has finite ramification. We denote by $\C_p^{\text{br}}$ the set of all elements of $\C_p$ of bounded ramification; clearly, $\oQp\subset\C_p^{\text{br}}$. A transcendental element $\alpha\in\C_p$ has bounded ramification if and only if the set of ramification indexes $\{e_n\}_{n\in\N}$ attached to a stacked  sequence $E\subset\oQp$ converging to $\alpha$ is bounded; in fact, by Theorem   \ref{DVRoverZp}, the integer $e$ such that $e=e_n$ for all $n$ sufficiently large is equal to $e(\Q_p(\alpha)|\Q_p)$.  

We remark that not all the transcendental elements $\alpha\in\C_p$ have bounded ramification. For example, according  to \cite{IovZah} there exist \emph{generic} transcendental elements $t\in\C_p$ for $\C_p$, that is, the completion of $\Q_p(t)$ is equal to $\C_p$. In particular, the value group of the unique valuation of $\O_{p,t}$ is equal to $\Q$, so the corresponding ramification index is $\infty$. Hence, by Proposition \ref{residuevaluerat}, $\Z_{p,t}$ has  value group equal to $\Q$  and therefore the  set of ramification indexes $\{e_n\}_{n\in\N}$ is unbounded. 

We show in Theorem \ref{prescribed residue field and value group} that given any algebraic extension $k$ of $\F_p$ and group $\Gamma$ such that $\Z\subseteq\Gamma\subseteq\Q$, there exists a transcendental element $\alpha\in\C_p$ such that $\Z_{p,\alpha}$ has residue field $k$ and value group $\Gamma$, provided that either $[k:\F_p]$ is infinite or $\Gamma$ is not discrete (this condition being necessary by Corollary  \ref{eof infinite}).
\end{Rem}

\begin{Lem}\label{maxunr}
Let $l$ be an infinite algebraic extension of $\Q_p$ such that $e(l|\Q_p)$ is finite. Then $l$ is contained in the maximal unramified extension $K^{\text{unr}}$ of a finite extension $K$ of $\Q_p$.
\end{Lem}
\begin{proof}
For each $n\in\N$, let $\Q_p^{(n)}$ be the compositum of all the extensions of $\Q_p$ of degree bounded by $n$. Clearly, $\oQp=\bigcup_{n\in\N}\Q_p^{(n)}$ and $\Q_p^{(n)}\subset \Q_p^{(n+1)}$ for each $n\in\N$. Since $\Q_p$ has only finitely many extensions of bounded degree, $\Q_p^{(n)}=\Q_p(t_n)$ for some $t_n\in\oQp$. Now, for each $n\in\N$, we let $\Q_p(t_n)\cap l=\Q_p(s_n)$ for some $s_n\in l$. Clearly, $l=\bigcup_{n\in\N}\Q_p(s_n)$ and $\Q_p(s_n)\subset \Q_p(s_{n+1})$, for each $n\in\N$. Since $\Gamma_{s_n}\subseteq\Gamma_{s_{n+1}}\subseteq\Gamma_l$ for each $n\in\N$ and $\Gamma_l$ is discrete by assumption, there exists $n_0\in\N$ such that $\Gamma_{s_n}=\Gamma_{s_{n_0}}$ for each $n\geq n_0$. Therefore, if $K=\Q_p(s_{n_0})$, then $s_n\in K^{\text{unr}}$ for each $n\geq n_0$, so that $l\subseteq K^{\text{unr}}$.
\end{proof}

The next lemma shows that a transcendental element $t$ of $\C_p$ with bounded ramification arise as the limit of sequences contained in the maximal unramified extension $K^{\text{unr}}$ of a finite extension $K$ of $\Q_p$. We don't know whether there exists a stacked sequence in $K^{\text{unr}}$ which converges to $t$.

\begin{Prop}\label{elements bounded ramification}
Let $t\in\C_p^{\text{br}}$. Then $t$ is the limit of a  sequence contained in the maximal unramified extension of a finite extension of $\Q_p$.
\end{Prop}

\begin{proof}
By \cite[Theorem 1]{IovZah}, the completion of $\widehat{\Q_p(t)}\cap\oQp$ is equal to $\widehat{\Q_p(t)}$. In particular, there exists a Cauchy sequence $E=\{s_n\}_{n\in\N}\subset\widehat{\Q_p(t)}\cap\oQp$ converging to $t$. Now, since $\Q_p(t)\subset\widehat{\Q_p(t)}$ and $\widehat{\Q_p(t)}\cap\oQp\subset \widehat{\Q_p(t)}$ are immediate extensions, it follows that $\widehat{\Q_p(t)}\cap\oQp$ has value group $\Gamma_t$ and residue field $k_t$. By Lemma \ref{maxunr}, $\widehat{\Q_p(t)}\cap\oQp$ is contained in the maximal unramified extension of a finite extension of $\Q_p$. The statement follows.
\end{proof}

The following result is not new, see for example \cite[Lemma 2]{Lampert}, the present proof is different because it employs the notion of stacked sequence.

\begin{Thm}\label{prescribed residue field and value group}
Let $k$ be an algebraic extension of $\F_p$ and $\Gamma$ a totally ordered group with $\Z\subseteq\Gamma\subseteq\Q$, such that either $[k:\F_p]$ or $[\Gamma:\Z]$ is infinite (the last condition is equivalent to $\Gamma$ being not discrete). Then there exists a transcendental element $\alpha\in\C_p$ such that $k_{\alpha}=k$ and  $\Gamma_{\alpha}=\Gamma$. In particular, $\Z_{p,\alpha}$ has residue field $k$ and value group $\Gamma$.
\end{Thm}
Note that, by Corollary \ref{eof infinite}, the last claim shows that $[k:\F_p]\cdot[\Gamma:\Z]=\infty$ is necessary.
\begin{proof}
Since $\overline{\F_p}$ is countable, we may suppose that $k=\bigcup_{n\in\N}k_n$, where $k_n$ is a finite extension of $\mathbb{F}_p$, $k_n\subseteq k_{n+1}$ and $k_0=\mathbb{F}_p$. Similarly, $\Gamma=\bigcup_{n\in\N}\Gamma_n$, where $\Gamma_n$ is a discrete group, $\Gamma_n\subseteq\Gamma_{n+1}$ and $\Gamma_0=\Z$. Let $f=[k:\F_p]$ and $e=[\Gamma:\Z]$; then, either $e$ or $f$ is infinite. Without loss of generality, we may assume that for each $n$, $[k_{n+1}:k_n][\Gamma_{n+1}:\Gamma_n]>1$.

For each $n\in\N$, there exists a local field  $K_n=\Q_p(s_n)$ with residue field $k_n$ and value group $\Gamma_n$. By induction, we may also assume that $K_n\subset K_{n+1}$. Let $\{\lambda_n\}_{n\in\N}\subset\Q$ be a strictly increasing sequence in $\Q$ which is unbounded and $\lambda_0<\delta_0=v(s_1-s_0)$.

We define now a sequence $E=\{t_n\}_{n\in\N}\subset\oQp$ such that for each $n\in\N$, $n\geq 1$, we have
\begin{itemize}
\item[i)] $\Q_p(t_n)=\Q_p(s_n)$;
\item[ii)] $(t_{n-1},\delta_{n-1}=v_p(t_n-t_{n-1}))$ is a minimal pair;
\item[iii)] $\delta_{n-1}>\lambda_{n-1}$.
\end{itemize}
In particular, $E$ is a stacked sequence by conditions i) and ii) and Cauchy by condition iii) and the assumption on $\{\lambda_n\}_{n\in\N}$.

We set $t_0=s_0\in\Q_p$, $t_1=s_1\not\in\Q_p$ and $\delta_0=v_p(t_1-t_0)$. Note that $(t_0,\delta_0)$ is a minimal pair. We proceed by induction on $n$. We assume that for all $m<n$ we have chosen $t_m\in\oQp$ such that conditions i), ii), iii) above are satisfied.

We now show how to choose $t_n$. We choose $a_n\in\Q_p$, $a_n\not=0$, such that
$$v_p(a_n)>\max\{\omega(t_{n-1})-v_p(s_n),\lambda_{n-1}-v_p(s_n)\}$$ We then set
$$t_n=a_ns_n+t_{n-1}.$$
Note that $\Q_p(t_n)\subseteq\Q_p(s_n)$, since by induction $\Q_p(t_{n-1})=\Q_p(s_{n-1})$ and the last field is contained in $\Q_p(s_n)$. Now, since $\delta_{n-1}=v_p(t_n-t_{n-1})>\omega(t_{n-1})$, it follows by Krasner's Lemma that $\Q_p(t_{n-1})\subseteq\Q_p(t_n)$. This containment and the fact that $s_n=\frac{t_n-t_{n-1}}{a_n}$ show that $s_n$ is in $\Q_p(t_n)$, so that $\Q_p(t_n)=\Q_p(s_n)$. Moreover, note also that $\delta_{n-1}>\lambda_{n-1}$. Hence, $E=\{t_n\}_{n\in\N}$ is a stacked sequence which is Cauchy, so $E$ converges to a transcendental element $\alpha$ of  $\C_p$. By Proposition \ref{residuevaluerat}, $\Z_{p,E}=\Z_{p,\alpha}$ has residue field $k$ and value group $\Gamma$, as wanted. By  Proposition \ref{ef transcendental elements Cp}, $\Z_{p,\alpha}$ is isomorphic to $\O_{p,\alpha}$, so it follows the $\Gamma_{\alpha}=\Gamma$ and $k_{\alpha}=k$.
\end{proof}
\begin{Rem}\label{discrete Gamma}
We remark that without condition iii) above in the proof of Theorem \ref{prescribed residue field and value group}, in general we may only conclude that there exists a stacked sequence $E\subset\oQp$ (which may not be Cauchy) such that the valuation domain $\Z_{p,E}$ has residue field $k$ and value group $\Gamma$. If instead $\Gamma$ is discrete by assumption, condition iii) is not necessary: in fact, there exists $n_0\in\N$ such that $\Gamma_n=\Gamma_{n_0}=\Gamma$ for all $n\geq n_0$; that is, $K_n=\Q_p(s_n)$ is an unramified extension of $K_{n_0}$ for all $n>n_0$. Hence, $E\subset \bigcup_{n\in\N}K_n$ is Cauchy, and so $\Z_{p,E}=\Z_{p,\alpha}$, where $\alpha\in\C_p^{\text{br}}$ is the transcendental limit of $E$. 
\end{Rem}

We close this section showing that the statement of \cite[Proposition 1]{IovZah} is wrong, namely, in general the completion of $\Q_p(X)$ with respect to a residually algebraic torsion extension $W$ of $\Z_p$ may not be a subfield of $\C_p$. The mistake is due to the fact that if $W=\Z_{p,E}$ for some pseudo-convergent sequence $E\subset\oQp$ of transcendental type, then $X$ is a pseudo-limit of $E$ with respect to $w$ and may not be a limit (that is, $E$ may not be Cauchy).

\begin{Prop}\label{completion rat Qp}
Let $W$ be a residually algebraic torsion extension of $\Z_p$ to $\Q_p(X)$. Then the completion $\widehat{\Q_p(X)}$ with respect to $W$ is (isomorphic to) a subfield of $\C_p$ if and only if there exists a transcendental element $\alpha$ in $\C_p$ such that $W=\Z_{p,\alpha}$.
\end{Prop}
\begin{proof}
By Theorem \ref{Epcvtranscdistinguished}, there exists a pseudo-convergent  sequence $E=\{s_n\}_{n\in\N}\subset\oQp$ of transcendental type such that $W=\Z_{p,E}$. 



Suppose that $\widehat{\Q_p(X)}\subseteq\C_p$. In particular, $X\in\C_p$ and so there exists a Cauchy sequence $F=\{t_n\}_{n\in\N}\subset\oQp$ which tends to $X$. Since $\Q_p(X)\subset\oQp(X)$ is an algebraic extension and $\C_p$ is algebraically closed, then also the completion of $\oQp(X)$ with respect to $\overline{W}=\overline{\Z_p}_E$ is contained in $\C_p$. Without loss of generality, we may suppose that the restriction of $v_p$ to $\oQp(X)$ is equal to $\overline w$. In particular, $\overline{w}(X-t_n)=v_p(X-t_n)\nearrow\infty$. Since $E$ is of transcendental type, for each $n$ there exists $m_0$ such that $\overline{w}(X-t_n)<\overline{w}(X-s_m)$ for each $m\geq m_0$. This shows that the gauge of $E$ tends to infinity, thus $E$ is Cauchy; in particular, $E$ converges to a transcendental element $\alpha\in\C_p$. Therefore, $W=\Z_{p,\alpha}$.

Conversely, let $W=\Z_{p,\alpha}$ for some transcendental element $\alpha\in\C_p$. Then, by Proposition \ref{ef transcendental elements Cp}, the completion $\widehat{\Q_p(X)}$ with respect to $\Z_{p,\alpha}$ is isomorphic to the completion of $\Q_p(\alpha)$ and therefore can be identified to a subfield of $\C_p$.
\end{proof} 
In particular, if $W=\Z_{p,E}$ for some stacked non-Cauchy sequence $E\subset\oQp$, then $\widehat{\Q_p(X)}$ is not contained in $\C_p$.

\subsection{Residually algebraic torsion extensions of $\Z_{(p)}$}\label{RatZpQX}

We now characterize residually algebraic torsion extensions of $\Z_{(p)}$ to $\Q(X)$. We remark that such a valuation domain may have an extension to $\Q_p(X)$ which is a residually algebraic extension of $\Z_p$  but is not torsion. For example, let $\alpha\in\oQp$ be transcendental over $\Q$, then $\Z_{(p),\alpha}$ is torsion but $\Z_{p,\alpha}$ is not (the one dimensional valuation overring of $\Z_{p,\alpha}$ is $\Q_p[X]_{(p_{\alpha}(X))}$, where $p_{\alpha}(X)$ is the minimal polynomial of $\alpha$ over $\Q_p$).

Given $\alpha\in\C_p$, we consider the following valuation domain of $\Q(X)$:
$$\Z_{(p),\alpha}=\{\phi\in\Q(X)\mid\phi(\alpha)\in \O_p\}$$
which is just the contraction to $\Q(X)$ of $\Z_{p,\alpha}$ considered in \S \ref{stackedseq}. Similarly, if $E=\{s_n\}_{n\in\N}\subset\oQp$ is a pseudo-convergent sequence of transcendental type, then we set 
$$\Z_{(p),E}=\{\phi\in\Q(X)\mid \phi(s_n)\in\oZp,\text{ for all sufficiently large }n\in\N \}$$
which is equal to $\Z_{p,E}\cap\Q(X)$.

The next proposition is analogous to Proposition \ref{ratpcvtransc}, and characterizes residually algebraic torsion extensions of $\Z_{(p)}$ to $\Q(X)$ in terms of pseudo-convergent sequences of $\oQp$ which are of transcendental type over $\Q$; clearly, every pseudo-convergent sequence of transcendental type of $\oQp$ belongs to this class. As a particular case, we find again part of the result of \cite[Theorem 2.5]{PerTransc}.

\begin{Prop}\label{ratQ(X)}
Let $p\in\PP$ and let $W$ be a residually algebraic torsion extension of $\Z_{(p)}$ to $\Q(X)$.  Then there exists a pseudo-convergent sequence $E\subset\oQp$ of transcendental type over $\Q$ such that $W=\Z_{(p),E}$. More precisely, let $e,f$ be the ramification index and residue field degree of $W$ over $\Z_{(p)}$, respectively. Let $\widehat{\Q(X)}$ be the completion of $\Q(X)$ with respect to the $W$-adic topology. Then the following conditions are equivalent:
\begin{enumerate}
\item $\widehat{\Q(X)}$ is a finite extension of $\Q_p$.
\item $X$ is algebraic over $\Q_p$.
\item $W=\Z_{(p),\alpha}$, for some $\alpha\in\oQp$ transcendental over $\Q$. 
\item $ef<\infty$.
\end{enumerate}
If either one of these conditions holds, then the sequence $E$ above is Cauchy and converges to $\alpha$ (and $E$ is therefore of algebraic type over $\Q_p$). Moreover, we have $\Gamma_w=\Gamma_{\alpha}$ and $k_w=k_{\alpha}$.

If $ef=\infty$, then $E\subset\oQp$ is of transcendental type over $\Q_p$ and $\Z_{(p),E}\subset\Z_{p,E}$ is an immediate extension.
\end{Prop}
\begin{proof}
Note that since $W$ is a torsion extension of $\Z_{(p)}$, the $p$-adic completion $\Q_p$ of $\Q$ is contained in $\widehat{\Q(X)}$ (see for example the arguments given in the proof of \cite[Corollary 2.6]{APZTheorem}).

If $\widehat{\Q(X)}$ is a finite extension of $\Q_p$ then clearly $X$ is algebraic over $\Q_p$, so 1. implies 2. If $X$ is algebraic over $\Q_p$, we may identify $X$ with some $\alpha\in\oQp$; $\Q_p(\alpha)$ is a finite extension of $\Q_p$, hence complete. So, $\widehat{\Q(X)}=\Q_p(\alpha)$. As in the proof of \cite[Theorem 2.5]{PerTransc} it follows easily that $W=\Z_{(p),\alpha}$. Therefore, 2. implies 3. 

If $W=\Z_{(p),\alpha}$ for some $\alpha\in\oQp$ transcendental over $\Q$, then by \cite[Proposition 2.2]{PerTransc}, $ef<\infty$, so 3. implies 4. Finally, 4. implies 1. by \cite[Lemma 2.4]{PerTransc} because $e(\widehat{W}\mid\Z_p)=e$ and $f(\widehat{W}\mid\Z_p)=f$.

Note that if $E\subset\oQp$ is a pseudo-convergent sequence such that $\Z_{(p),E}=\Z_{(p),\alpha}$, then by Lemma \ref{extension Vpa to completion} below we have $\Z_{p,E}=\Z_{p,\alpha}$, so by \cite[Proposition 5.3]{PS2} we have that $E$ is Cauchy and converges to $\alpha$.


The claims about the value group and residue field of $\Z_{(p),\alpha}$ follow by \cite[Proposition 2.2]{PerTransc}.
 
\medskip

If $ef=\infty$ then $X$ is transcendental over $\Q_p$ by the previous part of the proof; in particular, the field of rational functions $\Q_p(X)$ is contained in $\widehat{\Q(X)}$. If $\wW=\widehat{W}\cap\Q_p(X)$, then $\wW$ is a residually algebraic torsion extension of $\Z_p$ to $\Q_p(X)$, so by Theorem \ref{Epcvtranscdistinguished} there exists a stacked sequence $E\subset\oQp$ such that $\wW=\Z_{p,E}$ (by Lemma \ref{distinguishedpcv}, $E$ is a pseudo-convergent sequence of transcendental type, necessarily unbounded). Restricting down to $\Q(X)$ we get $W=\Z_{(p),E}$. Finally, since $W\subset\widehat{W}$ is an immediate extension, it follows that $\Z_{(p),E}\subset\Z_{p,E}$ is an immediate extension, too. Hence, the value group and residue field of $\Z_{(p),E}$ are the same as those of $\Z_{p,E}$, respectively (see Proposition \ref{residuevaluerat}).
\end{proof}

The following statement is the analogous of Proposition \ref{completion rat Qp} for residually algebraic torsion extensions of $\Z_{(p)}$ to $\Q(X)$.
\begin{Cor}\label{completion rat Q}
Let $W$ be a residually algebraic torsion extension of $\Z_{(p)}$ to $\Q(X)$. Then the completion $\widehat{\Q(X)}$ with respect to $W$ is (isomorphic to) a subfield of $\C_p$ if and only if  there exists $\alpha\in\C_p$, transcendental over $\Q$, such that $W=\Z_{(p),\alpha}$.
\end{Cor}
\begin{proof}
According to Proposition \ref{ratQ(X)}, when passing to the completion, either $X$ is algebraic over $\Q_p$ or $X$ is transcendental over $\Q_p$ and consequently either $\widehat{\Q(X)}\subset\oQp\subset\C_p$ or $\Q_p(X)\subset\widehat{\Q(X)}$, respectively. In the first case, $W=\Z_{(p),\alpha}$ for some $\alpha\in\oQp\subset\C_p$ transcendental over $\Q$. In the second case, $\widehat{\Q_p(X)}=\widehat{\Q(X)}$, where the completion of $\Q_p(X)$ is considered with respect to the valuation domain $\wW=\widehat{W}\cap\Q_p(X)$. In particular, by Proposition \ref{completion rat Qp}, we get that $\widehat{\Q(X)}\subseteq\C_p$ if and only if there exists a transcendental element $\alpha\in\C_p$ such that $W=\Z_{(p),\alpha}$.
\end{proof}

In particular, if $W=\Z_{(p),E}$ for some stacked non-Cauchy sequence $E\subset\oQp$, then $\widehat{\Q(X)}$ is not contained in $\C_p$.

The following result is the analogous of Theorem \ref{prescribed residue field and value group} for building residually algebraic torsion extensions $W$ of $\Z_{(p)}$ to $\Q(X)$ with prescribed residue field $k$ and value group $\Gamma$. Note that, contrary to that Theorem, now we are not assuming anymore that $[k:\F_p]\cdot[\Gamma:\Z]=\infty$.
\begin{Thm}\label{prescribed residue field and value group rat Q(X)}
Let $k$ be an algebraic extension of $\F_p$ and $\Gamma$ a totally ordered group such that $\Z\subseteq\Gamma\subseteq\Q$. Then there exists $\alpha\in\C_p$, transcendental over $\Q$, such that $\Z_{(p),\alpha}$ has residue field $k$ and value group $\Gamma$.
\end{Thm}
\begin{proof}
Let $e=[\Gamma:\Z]$ and $f=[k:\F_p]$. If $ef<\infty$, then it is well known that  there exists $\alpha\in\oQp$ transcendental over $\Q$ such that $\O_{p,\alpha}$ has residue field $k$ and value group $\Gamma$. Hence, by \cite[Proposition 2.2]{PerTransc}, $\Z_{(p),\alpha}$ is the desired extension of $\Z_{(p)}$.

If $ef=\infty$, then, by Theorem \ref{prescribed residue field and value group}, there exists a transcendental element $\alpha\in\C_p$ such that $\Z_{p,\alpha}$ has residue field $k$ and value group $\Gamma$. Clearly, $\Z_{p,\alpha}\cap\Q(X)=\Z_{(p),\alpha}$ is a residually algebraic torsion extension of $\Z_{(p)}$ to  $\Q(X)$. Moreover, by Proposition  \ref{ef transcendental elements Cp}, $\Z_{p,\alpha}=\Z_{p,E}$ for some stacked Cauchy sequence $E\subset\oQp$ which converges to $\alpha$. In particular, $\Z_{(p),\alpha}=\Z_{(p),E}$. By the last part of Proposition \ref{ratQ(X)}, $\Z_{(p),E}\subset\Z_{p,E}$ is an immediate extension, so $\Z_{(p),\alpha}$ has residue field $k$ and value group $\Gamma$.
\end{proof}


Now we are able to describe the DVRs of $\Q(X)$ which are residually algebraic   extensions of  $\Z_{(p)}$, for some $p\in\PP$.  We recall that every $\sigma\in G_{\Q_p}=\Gal(\oQp/\Q_p)$   extends uniquely to a continuous automorphism of $\C_p$, see \cite[\S 3]{APZclosed}. Given $\alpha,\beta\in \C_p$, we say that $\alpha,\beta$ are conjugate (over $\Q_p$) if there exists $\sigma\in G_{\Q_p}=\Gal(\oQp/\Q_p)$ such that $\sigma(\alpha)=\beta$; the orbit of an element $\alpha\in\C_p$ is finite if and only if $\alpha\in\oQp$ (see \cite[Remark 3.2]{APZclosed}). 

We prove first the following lemma.

\begin{Lem}\label{extension Vpa to completion}
Let $p\in\PP$ and $W$ a valuation domain of $\Q_p(X)$ such that  $W\cap \Q(X)=\Z_{(p),\alpha}$ for some $\alpha\in \C_p$. Then $W=\Z_{p,\alpha}$.
\end{Lem}
\begin{proof}
Let $n\geq0$ be an integer such that $p^n\cdot \alpha=\alpha_0\in\O_p$. The field isomorphism $X\mapsto\frac{X}{p^n}$ maps $\Z_{(p),\alpha}$ to $\Z_{(p),\alpha_0}$ and $\Z_{p,\alpha}$ to $\Z_{p,\alpha_0}$, respectively. Hence, in order to prove the statement, without loss of generality, we may assume  that $\alpha\in \O_p$.

Let $w$ be a valuation associated to $W$. We note first that since $X\in  \Z_{(p),\alpha}$, it follows that $w(X)\geq0$. Let $f\in W\cap\Q_p[X]$, say $f(X)=\sum_{i=0}^d \alpha_i X^i$. Then for $g(X)=\sum_{i=0}^d a_i X^i\in\Q[X]$ we have
$$w(f-g)\geq\min_{0\leq i\leq d}\{v_p(\alpha_i-a_i)+iw(X)\}.$$ 
Therefore, if we choose $a_i\in\Q$ sufficiently $v_p$-adically close to $\alpha_i$ for each $i=0,\ldots,d$, we have $w(f-g)\geq0$. In particular, $g\in W\cap \Q(X)=\Z_{(p),\alpha}$. The polynomial $h=f-g$ is in $\Z_p[X]$; therefore $f(\alpha)=h(\alpha)+g(\alpha)\in\O_p$, so that $f\in \Z_{p,\alpha}$. Therefore $W\cap \Q_p[X]\subseteq \Z_{p,\alpha}\cap\Q_p[X]$. Similarly, one can easily show that the other containment holds, so $W\cap \Q_p[X]=\Z_{p,\alpha}\cap\Q_p[X]$.  In the same way, $M_W\cap \Q_p[X]=M_{p,\alpha}\cap\Q_p[X]$, where $M_W$ and $M_{p,\alpha}$ are the maximal ideals of $W$ and $\Z_{p,\alpha}$, respectively.

Let now $\psi\in\Z_{p,\alpha}$; since $\Z_p[X]\subset \Q_p[X]\cap \Z_{p,\alpha}$, we may suppose that $\psi=\frac{f}{g}$, where $f,g\in\Q_p[X]\cap \Z_{p,\alpha}$. Clearly, $g(\alpha)\not=0$; then, there exists $n\in\N$, $n\geq 1$ and $c\in\Q_p$, $c\not=0$, such that $v_p(c)+v_p(g(\alpha)^n)=0$. We consider then the rational function $\psi^n=\frac{cf^n}{cg^n}=\frac{f_1}{g_1}$, which still is in $\Z_{p,\alpha}$. Note that $f_1\in \Z_{p,\alpha}\cap\Q_p[X]=W\cap\Q_p[X]$ and $g_1\in \Z_{p,\alpha}^*\cap\Q_p[X]=W^*\cap \Q_p[X]$, since $v_{p,\alpha}(f_1)\geq v_{p,\alpha}(g_1)=0$ (the ${}^*$ denotes the set of units of the valuation domains). In particular, 
$$w(f_1)\geq0=w(g_1)$$
which proves that $\psi^n\in W$. Since $W$ is integrally closed, it follows that $\psi\in W$. 
Hence, $\Z_{p,\alpha}\subseteq W$. The equality follows because both rings are extensions of $\Z_p$ to $\Q_p(X)$ and in the case $\alpha$ is algebraic over $\Q_p$, the one-dimensional valuation overring of $\Z_{p,\alpha}$ is non-unitary (i.e., $\Q_p[X]_{(q)}$, where $q\in\Q_p[X]$ is the minimal polynomial of $\alpha$).
\end{proof}

\begin{Cor}\label{DVRQ(X)rat}
Let $W$ be a DVR of $\Q(X)$ which is a residually algebraic  extension of $\Z_{(p)}$ for some $p\in\PP$. Then there exists $\alpha\in \C_p^{\text{br}}$, transcendental over $\Q$, such that $W=\Z_{(p),\alpha}$. The element $\alpha$ belongs to $\oQp$  if and only if the residue field extension $\Z/p\Z\subseteq W/M$ is finite.

Moreover, for $\alpha,\beta\in \C_p$, we have $\Z_{(p),\alpha}=\Z_{(p),\beta}$ if and only if there exists $\sigma\in G_{\Q_p}$ such that $\sigma(\alpha)=\beta$.
\end{Cor}
\begin{proof}
Let $f=[W/M:\Z/p\Z]$. If $f<\infty$, then the claim follows by \cite[Theorem 2.5]{PerTransc} and corresponds to the first case of Proposition \ref{ratQ(X)}: $W=\Z_{(p),\alpha}$, for some $\alpha\in\oQp$ which is transcendental over $\Q$. If $f=\infty$, then we are in the last case of Proposition \ref{ratQ(X)}, so $W=\Z_{(p),E}$, for some pseudo-convergent sequence in $\oQp$ of transcendental type. As in the proof of Proposition \ref{ratQ(X)}, we denote by $\widehat{W}$ the completion of $W$;  since the ramification index $e(W\mid\Z_{(p)})$ is finite, $\wW=\widehat{W}\cap\Q_p(X)$ is a residually algebraic torsion extension of $\Z_p$ to $\Q_p(X)$ which is a DVR, so by Theorem  \ref{DVRoverZp}, $\wW=\Z_{p,\alpha}$ for some $\alpha\in\C_p^{\text{br}}\setminus\oQp$. Hence, $W=\wW\cap\Q(X)=\Z_{(p),\alpha}$. Note that $\alpha$ is transcendental over $\Q_p$, hence also over $\Q$.

We prove now the final claim. Suppose there exists $\sigma\in G_{\Q_p}$ such that $\sigma(\alpha)=\beta$. If $\phi\in\Z_{(p),\alpha}$, then $\phi(\alpha)$ is defined and belongs to $\O_p$. In particular, $\sigma(\phi(\alpha))=\phi(\sigma(\alpha))=\phi(\beta)\in\oZp$. Hence, $\Z_{(p),\alpha}\subseteq\Z_{(p),\beta}$ and the other containment is proved in a symmetrical way.

Conversely, suppose that $\Z_{(p),\alpha}=\Z_{(p),\beta}$. By Lemma \ref{extension Vpa to completion}, it follows that $\Z_{p,\alpha}=\Z_{p,\beta}$. Note that the last two valuation domains are the contraction to $\Q_p(X)$ of the valuation domains $\oZp_{\alpha}=\{\phi\in\oQp(X)\mid \phi(\alpha)\in\O_p\},\oZp_{\beta}=\{\phi\in\oQp(X)\mid \phi(\alpha)\in\O_p\}$ of $\oQp(X)$, respectively. By \cite[Chapt. VI, §8, 6., Corollary 1]{Bourb}, there exists a $\Q_p(X)$-automorphism $\sigma$ of $\oQp(X)$ such that $\sigma(\oZp_{\alpha})=\oZp_{\beta}$. It is easy to check that $\sigma(\oZp_{\alpha})=\oZp_{\sigma(\alpha)}$. In particular, $\oZp_{\sigma(\alpha)}=\oZp_{\beta}$. If $\sigma(\alpha)-\beta\not=0$, let $c\in\oZp$ be such that $v_p(c)>v_p(\sigma(\alpha)-\beta)$. Let $a\in\oQp$ be such that $v_p(a-\sigma(\alpha))\geq v_p(c)$. Then the polynomial $\frac{X-a}{c}$ is in $\oZp_{\sigma(\alpha)}$ and not in $\oZp_{\beta}$, a contradiction.
\end{proof}

Note that, for a DVR $W$ as in the statement of Corollary \ref{DVRQ(X)rat}, there exists $\alpha\in \O_p\subset\C_p$ of bounded ramification such that $W=\Z_{(p),\alpha}$ if and only if $X\in W$. This last  condition occurs for example if $W$ is an overring of $\Z[X]$.

\vskip1cm
\section{Polynomial Dedekind domains}\label{PolDed}
In order to describe the family of Dedekind domains lying between $\Z[X]$ and $\Q[X]$,  we briefly recall the notion of integer-valued polynomials on algebras (see \cite{ChabPer, PerWerNontrivial}, for example). Let $D$ be an integral domain with quotient field $K$ and $A$ a torsion-free $D$ algebra. We embed $K$ and $A$ into the extended $K$-algebra $B=A\otimes_D K$, and this allows us to evaluate  polynomials over $K$ at elements of $A$. If $f\in K[X]$ and $a\in A$ are such that $f(a)\in A$, then we say that $f$ is integer-valued at $a$. In general, given a subset $S$ of $A$, we denote by
$$\Int_K(S,A)=\{f\in K[X] \mid f(s)\in A,\forall s\in S\}$$
the ring of integer-valued polynomials over $S$. We omit the subscript $K$ if $A=D$.

In our setting, let $\mathcal{O}=\prod_{p\in\PP}\O_p\subset\prod_{p\in\PP}\C_p$. Given $\alpha=(\alpha_p)\in  \prod_{p\in\PP}\C_p$ and $f\in\Q[X]$, then $f(\alpha)=(f(\alpha_p))$, which  is an element of $\prod_{p\in\PP}\C_p$. If $\uE=\prod_{p\in\PP}E_p$ is a subset of $\prod_p \C_p$, then
$$\IntQ(\uE,\mathcal{O})=\{f\in\Q[X]\mid f(\alpha)\in\mathcal{O},\forall \alpha\in\uE\}$$
that is, a polynomial $f$ is in $\IntQ(\uE,\mathcal{O})$ if $f(\alpha_p)\in \O_p$ for each $\alpha_p\in E_p$ and $p\in\PP$. By an argument similar to \cite[Remark 6.3]{ChabPer} there is no loss in generality to suppose that a subset of $\prod_{p\in\PP}\C_p$ is of the form $\prod_{p\in\PP}E_p$, when dealing with such rings of integer-valued polynomials.

We remark that we have the following representation for the ring $\IntQ(\uE,\mathcal{O})$ as an intersection of valuation overrings (see \cite[(2.2)]{PerDedekind}, for example):
\begin{equation}\label{IntQEZrepr}
\IntQ(\uE,\mathcal O)=\bigcap_{p\in\PP}\bigcap_{\alpha_p\in E_p}\Z_{(p),\alpha_p}\cap\bigcap_{q\in\mathcal{P}^{\text{irr}}}\Q[X]_{(q)}
\end{equation}
where  $\mathcal{P}^{\text{irr}}$ denotes the set of irreducible polynomials in $\Q[X]$. By \cite[Proposition 2.2]{PerTransc}, the valuation domain $\Z_{(p),\alpha_p}$ of $\Q(X)$ has rank one if and only if $\alpha_p$ is transcendental over $\Q$, and has rank $2$ otherwise (in the last case, note that necessarily $\alpha\in\oQp$).

A totally similar argument to \cite[Lemma 2.5]{PerDedekind} shows that, for $p\in\PP$, we have
$$(\Z\setminus p\Z)^{-1}(\IntQ(\uE,\mathcal O))=\IntQ(E_p,\O_p)=\{f\in\Q[X]\mid f(E_p)\subseteq \O_p\}$$

We also need to recall the following definition introduced in \cite{PerDedekind}.

\begin{Def}
We say that a subset $\uE$  of $\mathcal O$ is \emph{polynomially factorizable} if, for each $g\in\Z[X]$ and $\alpha=(\alpha_p)\in\uE$, there exist $n,d\in\Z$, $n,d\geq 1$ such that $\frac{g(\alpha)^n}{d}$ is a unit of $\mathcal O$, that is, $v_p(\frac{g(\alpha_p)^n}{d})=0$, $\forall p\in\PP$.
\end{Def}

The next theorem characterizes which rings of integer-valued polynomials $\IntQ(\uE,\mathcal{O})$ are Dedekind domains. Given   $p\in\PP$ and a subset $E_p$ of $\C_p$, we say that $E_p$ has finitely many $G_{\Q_p}=\Gal(\oQp/\Q_p)$-orbits if $E_p$ contains finitely many equivalence classes under the  relation of conjugacy over $\Q_p$ (we stress that $E_p$ may not necessarily contain a full $G_{\Q_p}$-orbit). By Corollary \ref{DVRQ(X)rat}, this condition holds  if and only if the set $\{\Z_{(p),\alpha_p}\mid \alpha_p\in E_p\}$ is finite.
Furthermore, if $E_p\subseteq\oQp$, then the number of $G_{\Q_p}$-orbits  is finite if and only if $E_p$ is a finite set.

\begin{Thm}\label{IntQEODed}
Let $\uE=\prod_{p\in\PP}E_p\subset\prod_p \C_p$. Then $\IntQ(\uE,\mathcal{O})$ is a Dedekind domain if and only if, for each prime $p$,  $E_p$ is a subset of $\C_p^{\text{br}}$ of transcendental elements over $\Q$ with finitely many $G_{\Q_p}$-orbits and $\uE$ is polynomially factorizable.

Moreover, if the above conditions hold, then the class group of $\IntQ(\uE,\mathcal{O})$ is isomorphic to the direct sum of the class groups $\IntQ(E_p,\O_p)$, $p\in\PP$, and if $E_p=\{\alpha_1,\ldots,\alpha_n\}$ where the $\alpha_i$'s are pairwise non-conjugate over $\Q_p$, then $\text{Cl}(\IntQ(E_p,\O_p))=\Z/e\Z\oplus\Z^{n-1}$, where $e$ is the greatest common divisors of the ramifications indexes of $\alpha_i$ over $\Q_p$.

In particular, assuming that $E_p$ is formed by pairwise non-conjugate elements over $\Q_p$ for each $p\in\PP$, $\IntQ(\uE,\mathcal{O})$ is a PID if and only if, $\uE$ is polynomially factorizable and for each $p\in\PP$, $E_p$ contains at most one element $\alpha_p\in \O_p\cap\C_p^{\text{br}}$, such that $\alpha_p$ is transcendental over $\Q$ and unramified over $\Q_p$.
\end{Thm}
\begin{proof}
Let $R=\IntQ(\uE,\mathcal{O})$.

Suppose that the above conditions on $\uE$ are satisfied. By \eqref{IntQEZrepr}, $R$ is equal to an intersection of DVRs. Moreover, $R$ has finite character, that is, for every non-zero $f\in R$, $f$ belongs to finitely many maximal ideals of the family of DVRs appearing in  \eqref{IntQEZrepr}: in fact, if $f(X)=\frac{g(X)}{n}$, for some $g\in\Z[X]$ and $n\in\Z$, $n\not=0$, then $f$ is divisible only by finitely many $q\in\mathcal{P}^{\text{irr}}$; since $\uE$ is polynomially factorizable, by \cite[Lemma 2.12]{PerDedekind}, the set $\{p\in\PP\mid \exists\alpha_p\in E_p, v_p(g(\alpha_p))>0\}$ is finite, so that $f$ belongs to finitely many maximal ideals of the family $\Z_{(p),\alpha_p}$, $\alpha_p\in E_p,p\in\PP$. Hence, $R$ is a Krull domain.

Suppose that $R$ is not a Dedekind domain. By \cite[Proposition 2.2]{Heitmann}, there exists a maximal ideal $M\subset R$ of height strictly bigger than one. If $M\cap\Z=(0)$, then, since $\Z[X] \subseteq R \subseteq \Q[X]$, it follows that $R_{\Z\setminus \{0\}} = \Q[X]$ and $2 \leq htM = ht(M_{\Z \setminus \{0\}}) \leq dim(\Q[X])=1$, a contradiction.
Hence, $M\cap \Z=p\Z$, for some $p\in\PP$. If we now localize at $p$, we have that $(\Z\setminus p\Z)^{-1}R=R_p=\IntQ(E_p,\O_p)$ which is a Dedekind domain by \cite[Theorem]{EakHei}. So $(\Z\setminus p\Z)^{-1}M\subset R_p$ cannot have dimension strictly bigger than one, a contradiction.

Conversely, suppose that $R$ is a Dedekind domain.  In  particular, for each $p\in\PP$, $(\Z\setminus p\Z)^{-1}R=R_p=\IntQ(E_p,\O_p)$ is a Dedekind domain, so $\{\Z_{(p),\alpha_p}\mid \alpha_p\in E_p\}$ is a finite set of DVRs (because $p$ is contained in only finitely many maximal ideals of these valuation overrings) which implies that $E_p$ is a subset of $\C_p^{\text{br}}$ of transcendental elements over $\Q$ and $E_p$ has finitely many $G_{\Q_p}$-orbits. Since every polynomial of $R$ is contained in only finitely many maximal ideals, it follows easily that $\uE$ is polynomially factorizable.

Finally, suppose that $R$ is a Dedekind domain. As in \cite[Lemma 2.14]{PerDedekind}, we have $\Cl(R)=\bigoplus_{p\in\PP}\Cl(R_p)$, where $R_p=\IntQ(E_p,\O_p)$ for $p\in\PP$. The  claim about the class group of $R_p$ follows by \cite[Proposition 2.10]{PerDedekind} or by \cite[Theorem]{EakHei}, since, for each $p\in\PP$, we are assuming that $E_p=\{\alpha_1,\ldots,\alpha_n\}$ is formed by pairwise non-conjugate elements over $\Q_p$. 

The claim about when $\IntQ(\uE,\O)$ is a PID is now straightforward.
\end{proof}
Let $\ohZ=\prod_{p\in\PP}\oZp$. In \cite[Theorem 2.17]{PerDedekind} we show that if $R$ is a Dedekind domain between $\Z[X]$ and $\Q[X]$ such that the residue fields of prime characteristic are finite fields, then $R=\IntQ(\uE,\ohZ)$, for some $\uE=\prod_p E_p\subset\ohZ$ such that $\uE$ is polynomially factorizable and for each $p\in\PP$, $E_p$ is a finite subset of $\oZp$  of transcendental elements over $\Q$.
Now, we are able to complete the classification of the Dedekind domains $R$, $\Z[X]\subset R\subseteq\Q[X]$, without any restriction on the residue fields.
\begin{Thm}\label{PolDedekind}
Let $R$ be a Dedekind domain such that $\Z[X]\subset R\subseteq \Q[X]$. Then $R$ is equal to $\Int_{\Q}(\uE,\mathcal{O})$, for some polynomially factorizable subset $\uE=\prod_{p\in\PP}E_p\subset\mathcal{O}$, such that, for each prime $p$,  $E_p\subset \O_p\cap\C_p^{\text{br}}$ is a finite set of transcendental elements over $\Q$.
\end{Thm}
\begin{proof}
Note first that, by \cite[Theorem 3.14]{PerPrufer}, no valuation overring of $W$ of $R$ can be a residually transcendental extension of $W\cap\Q$, since for such a valuation domain $W$, the domain $W\cap\Q[X]$ is not Pr\"ufer. Hence, for each prime ideal $P\subset R$ such that $P\cap\Z=p\Z$, $p\in\PP$, $R_P$ is a DVR of $\Q(X)$ which is a residually algebraic extension of $\Z_{(p)}$. By Corollary \ref{DVRQ(X)rat}, there exists $\alpha\in \O_p\cap\C_p^{\text{br}}$ such that $R_p=\Z_{(p),\alpha}$. Let $E_p$ be the subset of $\C_p^{\text{br}}$ formed by all such $\alpha_p$'s. Note that since $p$ is contained in only finitely many maximal ideals $P$ of $R$, it follows that $E_p$ is a finite set; moreover, each element of $E_p$ is transcendental over $\Q$, since $R_P$ is a DVR.  It now follows that 
$$R=\bigcap_{p\in\PP}\bigcap_{\substack{P\subset R\\P\cap\Z=p\Z}}R_P\cap\Q[X]=\bigcap_{p\in\PP}\IntQ(E_p,\O_p)=\IntQ(\uE,\mathcal{O}).$$
The rest of the statement follows by Theorem \ref{IntQEODed}.
\end{proof}
Finally, the next corollary describes the PIDs among the family of Dedekind domains between $\Z[X]$ and $\Q[X]$.
\begin{Cor}
Let $R$ be a PID such that $\Z[X]\subset R\subseteq \Q[X]$. Then $R$ is equal to $\Int_{\Q}(\uE,\mathcal{O})$, for some $\uE=\prod_{p\in\PP}E_p\subset\mathcal{O}$, such that, for each prime $p$,  $E_p$ contains at most one element $\alpha_p\in \O_p\cap\C_p^{\text{br}}$, such that $\alpha_p$ is transcendental over $\Q$ and unramified over $\Q_p$ and $\uE=\{\alpha=(\alpha_p)\}$ is polynomially factorizable.
\end{Cor}
\begin{proof}
By Theorem \ref{PolDedekind}, the ring $R$ is equal to $\Int_{\Q}(\uE,\mathcal{O})$, for some polynomially factorizable subset $\uE=\prod_{p\in\PP}E_p\subset\mathcal O$, such that, for each prime $p$,  $E_p\subset \O_p\cap\C_p^{\text{br}}$ is a set of transcendental elements over $\Q$ with finitely many $G_{\Q_p}$-orbits. Since by hypothesis  the class group of $R$ is trivial, it follows by Theorem \ref{IntQEODed} that for each $p\in\PP$, $E_p$ contains at most one element, which is transcendental over $\Q$ and unramified over $\Q_p$.
\end{proof}

\begin{Rem}
As we mentioned in the Introduction,  given a group $G$ which is the direct sum $G$ of a countable family of finitely generated abelian groups, there exists a Dedekind domain $R$ between $\Z[X]$ and $\Q[X]$ with class group $G$ (\cite[Theorem 3.1]{PerDedekind}). The domain $R$ of that  construction is equal to
$\IntQ(\uE,\mathcal{O})$ for some polynomially factorizable subset $\uE=\prod_{p\in\PP}E_p$, where $E_p$ is a finite subset of $\oQp$ of transcendental elements over $\Q$. In particular, $R$ has finite residue fields of prime characteristic (\cite[Theorem 2.17]{PerDedekind}); the reason is that the valuation overrings $\Z_{(p),\alpha_p}$ of $R$ in \eqref{IntQEZrepr} have finite residue fields precisely because $\alpha_p$ is chosen in $\oQp$ for each $p\in\PP$ (Proposition \ref{ratQ(X)}).

Now, by means of Theorem \ref{prescribed residue field and value group rat Q(X)}, with the same method used in \cite[Theorem 3.1]{PerDedekind}, we can build a Dedekind domain $R$, $\Z[X]\subset R\subseteq\Q[X]$, with prescribed class group $G$ as above and prescribed residue fields of prime characteristic, which can be finite or infinite algebraic extensions of the prime field $\F_p$, according to whether the the above elements $\alpha_p\in E_p\subset\C_p^{\text{br}}$ transcendental over $\Q$, are either algebraic or transcendental over $\Q_p$.
\end{Rem}

\subsection*{Acknowledgments}

The author wishes to thank the anonymous referee for his/her remarks.

\end{document}